\documentclass[12pt]{article}
\usepackage{amsthm,fullpage}
\usepackage{amsmath}
\usepackage{amssymb}
\usepackage{mathtools}
\usepackage{float}
\usepackage{algorithm}
\usepackage[noend]{algpseudocode}
\usepackage[toc,page]{appendix}
\usepackage{graphicx}
\usepackage{hyperref}
\usepackage[pdftex,dvipsnames,table]{xcolor}
\usepackage{soul} 
\usepackage[export]{adjustbox} 
\usepackage{subfig}

\usepackage{makecell}

\theoremstyle{plain}
\numberwithin{equation}{section}
\newtheorem{definition}{Definition}[section]
\newtheorem{theorem}{Theorem}[section]
\newtheorem{corollary}[theorem]{Corollary}
\newtheorem{lemma}[theorem]{Lemma}

\newtheorem{remark}{Remark}[section]

\DeclareMathOperator{\diag}{diag}
\DeclareMathOperator{\Tr}{Tr}

\DeclareMathOperator{\spec}{spec}

\DeclareMathOperator{\Rea}{Re}
\DeclareMathOperator{\Ima}{Im}

\newcommand{\hlgreen}[1]{{\sethlcolor{green}\hl{#1}}}

\title{{\textsf{Bounds for the extremal eigenvalues of gain Laplacian matrices}}}

\author{\textsf{M. Rajesh Kannan}\thanks{Department of Mathematics, Indian Institute of Technology Kharagpur, Kharagpur 721 302, India. Email: rajeshkannan@maths.iitkgp.ac.in, rajeshkannan1.m@gmail.com }\and  \textsf{Navish Kumar}  \thanks{Undergraduate Student, Department of Humanities and Social Sciences, Indian Institute of Technology Kharagpur, Kharagpur 721 302, India. Email: navish@iitkgp.ac.in, navish.iitkgp@gmail.com} \and  \textsf{Shivaramakrishna Pragada} \thanks{Undergraduate Student, Department of Aerospace Engineering,  Indian Institute of Technology Kharagpur, Kharagpur 721 302, India. Email: shivaram@iitkgp.ac.in, shivaramkratos@gmail.com }
}

\date{\today}
\begin{document}
\maketitle
\baselineskip=0.25in

\begin{abstract}
  A complex unit gain graph ($\mathbb{T}$-gain graph), $\Phi = (G, \varphi)$ is a graph where the function $\varphi$ assigns a unit complex number to each orientation of an edge of $G$, and its inverse is assigned to the opposite orientation. A $ \mathbb{T} $-gain graph $\Phi$ is balanced if the product of the edge gains of each cycle (with a fixed orientation)  is $1$. Signed graphs are special cases of $\mathbb{T}$-gain graphs.

 The adjacency matrix of $\Phi$, denoted by $ \mathbf{A}(\Phi)$ is defined canonically. The  gain Laplacian for $\Phi$ is defined as $\mathbf{L}(\Phi) = \mathbf{D}(\Phi) - \mathbf{A}(\Phi)$, where $\mathbf{D}(\Phi)$ is the diagonal matrix with diagonal entries are the degrees of the vertices of $G$. The minimum number of vertices (resp., edges) to be deleted from $\Phi$ in order to get a balanced gain graph is the frustration number (resp, frustration index). We show that frustration number and frustration index are bounded below by the smallest eigenvalue of $\mathbf{L}(\Phi)$.  We provide several lower and upper bounds for extremal eigenvalues of $\mathbf{L}(\Phi)$ in terms of different graph parameters such as the number of edges, vertex degrees, and average $2$-degrees. The signed graphs are particular cases of the $\mathbb{T}$-gain graphs, all the bounds established in this paper hold for signed graphs. Most of the bounds established here are new for signed graphs. Finally, we perform comparative analysis for all the obtained bounds in the paper with the state-of-the-art bounds available in the literature for randomly generated Erd\H{o}s-Re\'yni graphs.

 Some of the major highlights of our paper are the gain-dependent bounds, limit convergence of the bounds to the extremal eigenvalues, and optimal extremal bounds obtained by posing  optimization problems to achieve the best possible bounds.

\end{abstract}

\textbf{Keywords. }  Extremal eigenvalues, Frustration index, Frustration number,  Gain Laplacian matrix, Signed graphs.

{\bf Mathematics Subject Classification(2010):}    05C22(primary); 05C50(secondary).

\section{Introduction}
The study of matrices and eigenvalues associated with graphs has evolved over the past few decades. There has been a growing interest among researchers to study the adjacency, Laplacian, and normalized Laplacian matrices associated with gain graphs. The spectral properties of Laplacian $\mathbf{L}(\Phi)$ are leveraged to describe the graph-theoretic properties. Analyzing the extreme eigenvalues of the gain Laplacian, i.e., providing bounds for them and characterizing the types of graphs for which the equality holds, is an interesting problem to consider.

A signed graph $ \Psi=(G, \psi ) $ is a  graph $G$ with a signature function $\sigma: E(G) \rightarrow \{\pm 1\}$, where $E(G)$ is the edge set of $G$. Signed graphs are particular cases of $\mathbb{T}$-gain graphs. The smallest eigenvalue of the Laplacian of a signed graph $\mathbf{L}(\Psi)$ has been shown to be an excellent measure of the graph frustration, that is, the smallest number of vertices to be deleted from $\Psi$ in order to get a balanced graph \cite{belardo2014balancedness,belardo2020balancedness}. As is extensively surveyed in \cite{aref2019signed}, the frustration index is a key to frequently stated problems in many different fields of research \cite{doreian2015structural,dovslic2007computing,iacono2010determining,kasteleyn1963dimer}. In biological networks, optimal decomposition of a network into monotone subsystems is made possible by calculating the frustration index of the underlying signed graph \cite{iacono2010determining}. In physics, the frustration index provides the ground state of atomic magnet models \cite{hartmann2011ground,kasteleyn1963dimer}. In international relations, the dynamics
of alliances and enmities between countries can be investigated using the frustration
index \cite{doreian2015structural}. The frustration index can also be used as an indicator of network bi-polarisation in practical examples involving financial portfolios. For instance, some low-risk portfolios are shown to have an underlying balanced signed graph containing negative edges \cite{harary2002signed}. In chemistry, bipartite edge frustration can be used as a stability indicator of carbon allotropes known as fullerenes \cite{dovslic2005bipartivity, dovslic2007computing}.

In \cite{aref2020detecting}, the authors proposed programming models for optimal partitioning of a signed graph into cohesive groups. They tackle the intensive computations of dense signed networks by providing upper and lower bounds for the frustration index. This is a scenario where our optimal extremal bounds will be useful to close the gap between the two bounds, thereby returning the vertices' optimal partitioning. For more works in this direction we refer to \cite{aref2020multilevel,aref2020analyzing}. There are also lots of applications of the optimal extremal bounds for the Laplacian in the field of network science and analysis \cite{coscia2021atlas,estrada2005spectral}. As the networks' size grows, computing exact eigenvalues becomes intractable (for e.g., Twitter, Facebook user networks). In these types of scenarios, optimal extremal bounds will provide the best solutions one can hope for. These are the places where we draw our motivation for studying optimal extremal bounds.

The eigenvalues of the adjacency, Laplacian, and signless Laplacian matrices reveal several combinatorial information about the underlying graphs viz., number of edges, number of spanning trees, connectedness, bipartiteness, and many more. For more details, we refer to \cite{brouwer2011spectra,intro_graph_spectra_2009,signless-genesis}. In \cite{fallat-fan-laa2012}, the authors introduced the notion of vertex bipartiteness and the edge bipartiteness for graphs and shown that the smallest eigenvalue of the signless Laplacian matrix gives a lower bound for these quantities. Further, in \cite{belardo2014balancedness, belardo2020balancedness},  the authors studied the notion of frustration index and frustration number of signed graphs and $\mathbb{T}_4$-gain graphs and proved that the least eigenvalue of the signless Laplacian of the signed and  $\mathbb{T}_4$-gain graphs provides a lower bound for these quantities. In Section \ref{sec:frustration-bounds}, we define the notion of the frustration index and the frustration number for the $\mathbb{T}$-gain graphs, and  show that both of them are bounded below by the smallest eigenvalue of the gain Laplacian.  These notions are studied in the same spirit of the classical idea of algebraic connectivity and its connection with the second smallest eigenvalue of the Laplacian studied by Fiedler. Thus, the smallest eigenvalues of the signless Laplacian and the gain Laplacian matrices are interesting objects to consider.

In \cite{de2011smallest}, the authors established several fascinating bounds for the smallest eigenvalue of the signless Laplacian matrices. One of this article's main objectives is to extend these bounds for the complex unit gain graphs. All of our bounds depend on the gain of the underlying graph. Let $\Phi = (G,\varphi)$ be a $\mathbb{T}$-gain graph with $n$ vertices and $m$ edges, and let $a_{ij}$ be the gain of the edge connecting the vertices $v_i$ and $v_j$. Define \[a(\Phi) = \frac{1}{m}\bigg(\sum\limits_{\substack{v_i \sim v_j,\\ i<j}}(1-\Rea(a_{ij}))\bigg)~~~~~ \mbox{and}~~~~~ b(\Phi) = \frac{1}{m}\bigg(\sum\limits_{\substack{v_i \sim v_j,\\ i<j}}(1-\Ima(a_{ij}))\bigg).\] Let $\lambda_1(\Phi)$ be the smallest eigenvalue of the gain Laplacian matrix $L(\Phi)$ (defined in Section \ref{sec:prelim}). The main results of  Section \ref{sec:least-eig-bounds} are the bounds for $\lambda_1(\Phi)$ proved in Theorem \ref{chromres-2} and Theorem \ref{chro-ima}. Such bounds are further optimized for the bipartite graphs (Theorem \ref{bipart-refine}). Also, we show that these are the best possible bounds that can be obtained using this proof strategy.

%
%
%


Using Gershgorin's theorem, we propose novel upper bounds for the largest eigenvalue $\lambda_n(\Phi)$ of the gain Laplacian in terms of arbitrary invertible diagonal matrices (Theorem \ref{theorem:avg-2-degree}).  Besides this, we propose a couple of bounds for $\lambda_n(\Phi)$ in terms of powers and traces of the gain Laplacian matrix itself, and show the limit convergence of these bounds to $\lambda_n(\Phi)$ (Theorem \ref{thm:diag-elem} and Theorem \ref{trace-bound}). Finally, we devote the last section to analyze all the obtained bounds in the paper and compare them with the state-of-the-art bounds available in the literature for randomly generated Erd\H{o}s-Re\'yni graphs.

The outline of this paper is as follows:  In Section \ref{sec:prelim}, we include some needed known results for  graphs and matrices.  We prove, in Section \ref{sec:frustration-bounds}, a lower bound for the frustration index/number, in terms of least eigenvalue of $\mathbf{L}(\Phi)$. In Section \ref{sec:least-eig-bounds}, we derive bounds for the least eigenvalue of $\mathbf{L}(\Phi)$ in terms of the chromatic number and the edge gains. In Section \ref{sec:largest-eig-bounds}, we establish bounds for the largest eigenvalue of $\mathbf{L}(\Phi)$. In section \ref{sec:comparison}, we perform comparative analysis for all the obtained bounds with each other.

\section{Preliminaries}\label{sec:prelim}

Let $ G $ be a simple undirected graph. An oriented edge from the vertex $ v_i $ to the vertex $ v_j $ is denoted by $ \overrightarrow{e}_{ij} $. For each undirected edge $ e_{ij}\in E(G) $, there is a pair of oriented edges $ \overrightarrow{e}_{ij} $ and $ \overrightarrow{e_{ji}} $. The collection $ \overrightarrow{E}(G):=\{ \overrightarrow{e}_{ij},\overrightarrow{e_{ji}}: e_{ij}\in E(G)\} $ is  the \emph{oriented edge set associated with $ G $}. Given a group  $\mathfrak{G}$ and a graph $G$, the $\mathfrak{G}$-gain graph is defined as follows:  For each oriented edge $e_{ij}$ assign a value (the gain of the edge $e_{ij}$) $g$ from $\mathfrak{G}$ and assign $g^{-1}$ to the orientated edge $e_{ji}$. Gain graphs were widely studied in \cite{gain-genesis1, zaslavsky1989biased}.  If $\mathfrak{G} = \mathbb{T}$, where $ \mathbb{T}=\{ z\in \mathbb{C}: |z|=1\} $, then $\mathfrak{G}$-gain graphs are called  complex unit gain graphs. Precisely,  a \emph{complex unit gain graph (or $ \mathbb{T} $-gain graph)} on a simple graph $ G $ is an ordered pair $ (G, \varphi) $, where the gain function $ \varphi: \overrightarrow{E}(G) \rightarrow \mathbb{T} $ is a mapping  such that $ \varphi( \overrightarrow{e}_{ij}) =\varphi(\overrightarrow{e_{ji}})^{-1}$, for every $ e_{ij}\in E(G) $. A $ \mathbb{T} $-gain graph $ (G, \varphi) $ is  denoted by $ \Phi $. The \emph{ adjacency matrix} of  a $ \mathbb{T} $-gain graph $ \Phi=(G, \varphi)$ is  a Hermitian  matrix, denoted by $ \mathbf{A}(\Phi)$ and its $ (i,j)th $ entry is defined as follows:
$$a_{ij}=\begin{cases}
	\varphi(\overrightarrow{e_{ij}})&\text{if } \mbox{$v_i\sim v_j$},\\
	0&\text{otherwise.}\end{cases}$$
The spectrum and the spectral radius of $ \Phi $ are the spectrum and the spectral radius of $ \mathbf{A}(\Phi) $  and denoted by $ \spec(\Phi) $ and $ \rho(\Phi) $, respectively. Let $(G, 1)$ and $(G, -)$ denote the gain graphs with all the edge gains are equal to $1$ and $-1$, respectively. For a graph $G$ on $n$ vertices, the eigenvalues of $(G,-)$ are denoted by $\lambda_1(-) \leq \lambda_2(-) \leq \dots \leq \lambda_n(-)  $. For more details about the notion of $\mathbb{T}$-gain graphs, we refer to \cite{reff2012spectral,reff2016oriented,wang2018determinant,zaslavsky1989biased}.

The degree of the vertex $v_j$ is denoted by $d_j$. By slight abuse of notation, we write $\varphi(e_{ij})$ as only $\varphi_{ij}$, to represent the gain on edge $e_{ij}$. We define a diagonal matrix $\mathbf{D}(G) = \diag(d_1, d_2, ..., d_n )$, where $d_i$ is the degree of vertex $v_i$ in the underlying graph $G$. The  Laplacian matrix $\mathbf{L}(\Phi)$ is defined as $\mathbf{L}(\Phi) = \mathbf{D}(G)-\mathbf{A}(\Phi)$, where $\mathbf{D}(G) = \diag(d_1, d_2, ..., d_n )$ is a diagonal matrix and $d_i$ is the degree of vertex $v_i$ in the underlying graph $G$. It is clear from the above definition that $\mathbf{L}(\Phi)$ is Hermitian and positive semi-definite.

The gain of a cycle (with some orientation) $C = v_1 v_2 \ldots v_l v_1$, denoted by $\varphi(C)$, is defined as the product of the gains of its edges, that is $$\varphi(C) = \varphi(e_{12}) \varphi(e_{23}) \cdots \varphi(e_{(l-1)l}) \varphi(e_{l1}).$$
A cycle $C$ is said to be neutral if $\varphi(C) = 1$, and a gain graph is said to be balanced if all its cycles, if any, are neutral. For a cycle $C$ of $G$, we denote the real part of the gain of $C$ by $\Re(\varphi(C))$, and it is independent of the orientation.
A function from the vertex set of $G$ to the complex unit circle $\mathbb{T}$ is called a switching function. We say that, two gain graphs $\Phi_1 = (G, \varphi_1)$ and $\Phi_2 = (G, \varphi_2)$ are switching equivalent, written as $\Phi_1 \sim \Phi_2$, if there is a switching function $\zeta: V \to \mathbb{T}$ such that $$\varphi_2(e_{ij})=\zeta(v_i)^{-1}\varphi_2(e_{ij})\zeta(v_j).$$
The switching equivalence of two gain graphs can be defined in the following equivalent way: Two gain graphs $\Phi_1 = (G, \varphi_1)$ and $\Phi_2 = (G, \varphi_2)$ are switching equivalent, if there exists
a diagonal matrix $\mathbf{D}_{\zeta}$ with diagonal entries from $\mathbb{T}$, such that
\begin{align}\label{eq: switching equi}
	\mathbf{A}(\Phi_2) = \mathbf{D}_{\zeta}^{-1}\mathbf{A}(\Phi_1)\mathbf{D}_{\zeta}.
\end{align}Switching equivalence preserves connectivity and balancedness. The least Laplacian eigenvalue has a special role in the spectral theory of gain graphs. In fact, if the least eigenvalue is zero, then $\Phi = (G, \varphi)$ is switching equivalent to $(G, 1)$, and $\mathbf{L}(\Phi)$ is similar to $\mathbf{L}(G)$.  Similarly, if $\Phi$ is switching equivalent to $(G, -)$, then $\mathbf{L}(\Phi)$ is similar to $\mathbf{Q}(G) =\mathbf{D}(G) +\mathbf{A}(G)$, the signless Laplacian of $G$, and then we have the signless Laplacian theory of (usual) graphs \cite{kannan2020normalized, reff2012spectral}.

The following lemma will be useful.
\begin{lemma}[{\cite[Lemma 5.3]{reff2012spectral}}]   \label{lemma:unorm hermitian}
	Let $\Phi$ be a gain graph on $n$ vertices and $\mathbf{x}^T = (x_1 , x_2 , \dots, x_n) \in \mathbb{C}^n$ be a row vector. Then
	\begin{equation*}
		\mathbf{x}^*\mathbf{L}(\Phi)\mathbf{x} = \sum\limits_{\substack{v_i \sim v_j,\\ i<j}} |x_i-a_{ij} x_j|^2.
	\end{equation*}
\end{lemma}
\begin{remark}
Hereafter, to simplify the notation,  we write  $\sum\limits_{\substack{v_i \sim v_j}} |x_i-a_{ij} x_j|^2$ instead of  $\sum\limits_{\substack{v_i \sim v_j,\\ i<j}} |x_i-a_{ij} x_j|^2$, if there is no confusion.
\end{remark}
\begin{theorem}[{\cite[Corollary 4.5]{reff2012spectral}}]\label{theorem:upperbounds-beta_n}
	Let $\Phi = (G,\varphi)$ be a complex unit gain graph on $n$ vertices.  Let  $m_i = \frac{\sum_{v_i \sim v_j} d_j}{d_i}$ denote the average $2$-degree of the vertex $v_i$. Then,
	\begin{enumerate}
		\item[(i)] $\lambda_n(\Phi) \leq 2\Delta.$
		\item[(ii)] $\lambda_n(\Phi) \leq \max\limits_{v_i \sim v_j}\{d_i+m_i\}.$
		\item[(iii)] $\lambda_n(\Phi) \leq \max\limits_{v_i \sim v_j}\{d_i+d_j\}.$
		\item[(iv)]$\lambda_n(\Phi) \leq \max\limits_{v_i \sim v_j}\{(d_i (d_i+m_i)+d_j(d_j+m_j))/(d_i+d_j)\}.$
	\end{enumerate}Furthermore, equality holds if and only if $\Phi \sim (G, -)$.
\end{theorem}


The following theorem provides the variational formulation for the eigenvalue problem of the Hermitian matrices.
\begin{theorem}[Courant-Fischer Theorem, \cite{meyer2000matrix}] \label{theorem:CF}
	Let $\bf H$ be an $n \times n$ Hermitian matrix with eigenvalues $\lambda_1 \leq \lambda_2 \leq \ldots \leq \lambda_n$. For an integer $k \: (1 \leq k \leq n)$, we have
	
	$$\lambda_k = \max\limits_{\mathbf{x}^{(1)},\mathbf{x}^{(2)}, \ldots, \mathbf{x}^{(k-1)} \in \mathbb{C}^n} \: \min\limits_{\substack{\mathbf{x} \perp \mathbf{x}^{(1)},\mathbf{x}^{(2)}, \ldots, \mathbf{x}^{(k-1)}; \\ \mathbf{x} \neq 0; \\ \mathbf{x}  \in \mathbb{C}^n}} \frac{\mathbf{x}^*\mathbf{H}\mathbf{x}}{\mathbf{x}^*\mathbf{x}},$$
	
	$$\lambda_k = \min\limits_{\mathbf{x}^{(k+1)},\mathbf{x}^{(k+2)}, \ldots,
		\mathbf{x}^{(n)} \in \mathbb{C}^n} \max\limits_{\substack{\mathbf{x} \perp \mathbf{x}^{(k+1)},\mathbf{x}^{(k+2)}, \ldots, \mathbf{x}^{(n)}; \\ \mathbf{x} \neq 0; \\ \mathbf{x}  \in \mathbb{C}^n}} \frac{\mathbf{x}^*\mathbf{H}\mathbf{x}}{\mathbf{x}^*\mathbf{x}},$$where $\{\mathbf{x}^{(1)},\mathbf{x}^{(2)}, \ldots, \mathbf{x}^{(k-1)}\}$ and $\{ \mathbf{x}^{(k+1)},\mathbf{x}^{(k+2)}, \ldots, \mathbf{x}^{(n)} \}$ are linearly independent sets.
\end{theorem}

\begin{corollary}\label{corollary:max-diag-element}
	Let $\mathbf{H} = (h_{ij}) \in \mathbb{C}^{n\times n}$ be a  Hermitian matrix with eigenvalues $\lambda_1 \leq \dots \leq \lambda_n$. Then $\lambda_n \geq \max\limits_{i \in \{1,2,\dots,n\}}\{h_{ii}\}$
\end{corollary}

And, finally we provide the well known Gershgorian theorem.

\begin{theorem}[{\cite{gershgorin1931uber}}]
	Every eigenvalue of the $n \times n$ complex matrix $\mathbf{A} = (a_{ij})$ lies within at
	least one of the Gershgorin disks $D_i$, where
	\begin{align*}
		D_i = \{z_i \in \mathbb{C} \mid |z_i-a_{ii}| < \sum_{j \neq i} \vert a_{ij} \vert\}.
	\end{align*}
\end{theorem}

\section{ Frustration number and Frustration index for gain graphs}   \label{sec:frustration-bounds}

In this section, we define the notions of the frustration number and the frustration index for complex unit gain graphs. We provide bounds for the  above quantities in terms of the smallest eigenvalue of the Laplacian matrix of the complex unit gain graph. 	

For a gain graph $\Phi$, the frustration number
(resp. frustration index), denoted by $\nu(\Phi)$ (resp. $\epsilon(\Phi)$), is the minimum number of vertices (resp. edges) to be deleted such that the resultant gain graph
is balanced. Note that $\nu(\Phi) \leq \epsilon(\Phi)$ is evident.  It is well known that computing the frustration index is an NP-hard problem in general.  Let $0 \leq \lambda_1(\Phi) \leq \lambda_2(\Phi) \leq \ldots \leq \lambda_n(\Phi)$ be the eigenvalues of $\mathbf{L}(\Phi)$.

In the next theorem, we prove that $\lambda_1(\Phi) \leq \nu(\Phi)$. However, before proving it we need some additional notation. If $\Phi$ is unbalanced, then there exists a gain subgraph $S \subset \Phi$, with $|S| = \nu(\Phi)$, such that $\Phi \setminus S$ is balanced. Observe also that in the worst case $\Phi \setminus S = (K_2,\varphi)$, so $|\Phi \setminus S|\geq 2$  and, consequently, $|S| \leq |\Phi|-2$.

The proof of the following theorem is similar to that of  \cite[Theorem 3.2]{ belardo2020balancedness}.
\begin{theorem}\label{frus-bound-main}
	Let $\Phi$ be a $\mathbb{T}$-gain graph of order $n$. Then $\lambda_1(\Phi) \leq \nu(\Phi)$, and $\lambda_1(\Phi) \leq \epsilon(\Phi)$.
\end{theorem}
\begin{proof}
	Let $\nu(\Phi)=k$ and $|\Phi|=n$. If so, there exists $S \subset \Phi$ with $|S| = k \leq n-2$ such that $\Phi \setminus S$ is balanced. Let $\zeta$ be the switching function on $\Phi \setminus S$ which switches the balanced gain graph $\Phi \setminus S$ to its underlying graph. Let us define the following vector $\mathbf{x} = (\zeta,0,\dots,0)^T$ such that the length of string of zeroes at tail is $|S| = k$, so $\mathbf{x}^\star\mathbf{x} = |\Phi \setminus S|$. By the definition of switching function, we have $|x_i - a_{ij}x_j|^2 = 0$ whenever $v_i\sim v_j \in E(\Phi \setminus S)$, also $|x_i - a_{ij}x_j|^2 = 1$ for any edge in $E(S,\Phi \setminus S)$ and $|x_i - a_{ij}x_j|^2 = 0$ for all edges in $E(S,S)$. Now,
	\begin{align*}
		\lambda_1(\Phi) &\leq \frac{\sum_{v_i \sim v_j} |x_i-a_{ij} x_j|^2}{\mathbf{x}^\star \mathbf{x}}\\
		&= \frac{\sum_{v_i \sim v_j, v_i \in S, v_j \in \Phi \setminus S} |x_i-a_{ij} x_j|^2}{|\Phi \setminus S|}\\
		&\leq \frac{|S||\Phi \setminus S|}{|\Phi \setminus S|} = k.
	\end{align*}
\end{proof}
The smallest eigenvalue $\lambda_1(\Phi)$ of $\Phi$ being the lower bound for both frustration number $\nu(\Phi)$ and frustration index $\epsilon(\Phi)$ is called the\emph{ algebraic frustration} of $\Phi$.

Note that,  \cite[Theorem 3.1]{belardo2014balancedness}  for the signed graphs, and \cite[Proposition 3.1]{ belardo2020balancedness} for the $\mathbb{T}_4$-gain graphs follow from Theorem \ref{frus-bound-main}.

\section{ Bounds for the least eigenvalue of gain Laplacian}  \label{sec:least-eig-bounds}

We start by providing bounds in terms of the chromatic number of the underlying graph $G$, which are gain dependent as well, and later prove degree dependent bounds at the end of the section.

The lemma below is crucial for the inception of gain dependent quantity, $a_\theta(\Phi)$, which is the gateway to extend the inequality from the \cite[Theorem $2.11$]{de2011smallest} in context of gain graphs.

\begin{definition}\label{def:a-theta-phi}
	For  a unit complex gain graph $\Phi = (G, \varphi)$  with $n$ vertices and $m$ edges, with $m\geq1$, define $a_{\theta}(\Phi) = \frac{1}{m}\bigg(\sum\limits_{\substack{v_i \sim v_j}}(1-\Rea(a_{ij}e^{\mathbf{i}\theta}))\bigg)$. Then $a(\Phi) = a_0(\Phi)$ and $b(\Phi) = a_{-\frac{\pi}{2}}(\Phi)$.
\end{definition}

If $\Phi_1 \sim \Phi_2$, then $a(\Phi_1)$ need not be equal to $a(\Phi_2).$ Consider the complete graph on three vertices, with the gains $\varphi_1$ and $\varphi_2$ are given by the following adjacency matrices:
 $$ A(\Phi_1) = \begin{bmatrix}
 	0 & 1 & 1\\
 	1 & 0 & 1 \\
 	1 & 1 & 0
 \end{bmatrix},$$ and $$ A(\Phi_2) = \begin{bmatrix}
 	0 & 1 & -1\\
 	1 & 0 & -1 \\
 	-1 & -1 & 0
 \end{bmatrix}.$$
 Now, $$  \begin{bmatrix}
 	0 & 1 & 1\\
 	1 & 0 & 1 \\
 	1 & 1 & 0
 \end{bmatrix} = \begin{bmatrix}
 	1 & 0 & 0\\
 	0 & 1 & 0 \\
 	0 & 0 & -1
 \end{bmatrix}\begin{bmatrix}
 	0 & 1 & -1\\
 	1 & 0 & -1 \\
 	-1 & -1 & 0
 \end{bmatrix}\begin{bmatrix}
 	1 & 0 & 0\\
 	0 & 1 & 0 \\
 	0 & 0 & -1
 \end{bmatrix}.$$
 Thus $\Phi_1$ and $\Phi_2$ are switching equivalent, but $a(\Phi_1) = 0$ and $a(\Phi_2) = \frac{4}{3}$.

Given two nonempty subsets $V_1$ and $V_2$ of $V(G)$, let $e(V_1, V_2)$ denote the number edges between  $V_1$ and $V_2$ in $G$.
\begin{lemma}\label{chromres-5}
	Let $\Phi = (G, \varphi)$ be a unit complex gain graph with $n$ vertices and $m$, with $m\geq1$, edges and chromatic number $\chi$. Let  $V_1, V_2,\dots,V_\chi$ be the color classes of the underlying graph of $\Phi$. Then the following holds:
	\begin{align*}
		\lambda_{1}(\Phi) &\leq \frac{2m}{n}\bigg(a(\Phi) - \frac{(a(\Phi)-1)(\gamma^2+1) -2\gamma(a_{\theta}(\Phi)-1)}{\chi + \gamma^2 -1}\bigg),
	\end{align*}
	for all $\gamma \in \mathbb{R}$ and for all $ \theta \in [0,2\pi].$
\end{lemma}
\begin{proof}
	For $1\leq k \leq \chi$, define the vector $\mathbf{x}_k^T = (x_1,\dots,x_n)$ as follows:
	$$ x_i =  \begin{cases}
	a\quad \text{if} \ i \in V_k, \\
	be^{\mathbf{i}\theta} \quad \text{otherwise,}
	\end{cases} $$
	where $b \neq 0$. It is easy to see  that $||\mathbf{x}_k||^2 = (a^2-b^2)|V_k| + nb^2$, and
	$$	\sum_{v_i \sim v_j} |x_i - a_{ij}x_j|^2 = \sum_{v_i \sim v_j,v_i\in V_k,v_j \notin V_k} |a - a_{ij}e^{\mathbf{i}\theta}b|^2 +  \sum_{v_i \sim v_j,v_i\notin V_k,v_j \notin V_k} b^2|1 - a_{ij}|^2. $$
	
	Now, whenever $a_{ij} \neq 0$, we have $	|-1 + a_{ij}|^2 = 2(1-\Rea(a_{ij}))$, and
	\begin{align*}
		|a - a_{ij}e^{\mathbf{i}\theta}b|^2 &= (a-b)^2 + 2ab(1 - \Rea(a_{ij}e^{\mathbf{i}\theta})).
	\end{align*}
	Thus							
	$$	\sum_{v_i \sim v_j} |x_i - a_{ij}x_j|^2 = \sum_{v_i \sim v_j,v_i\in V_k,v_j \notin V_k} \big((a-b)^2 + 2ab(1 - \Rea(a_{ij}e^{\mathbf{i}\theta}))\big) + \sum_{v_i \sim v_j,v_i\notin V_k,v_j \notin V_k} 2b^2(1-\Rea(a_{ij})) , $$ and hence
	\begin{gather*}	
		\sum_{v_i \sim v_j} |x_i - a_{ij}x_j|^2 = (a-b)^2 e(V_k,V\setminus V_k) + \sum_{v_i \sim v_j,v_i\in V_k,v_j \notin V_k}2ab(1-\Rea(a_{ij}e^{\mathbf{i}\theta})) \\+ \sum_{v_i \sim v_j,v_i\notin V_k,v_j \notin V_k} 2b^2(1-\Rea(a_{ij})).
	\end{gather*}	
	
	Since $$\sum_{v_i \sim v_j,v_i\notin V_k,v_j \notin V_k} (1-\Rea(a_{ij})) = ma(\Phi) - \sum_{v_i \sim v_j,v_i\in V_k,v_j \notin V_k}(1-\Rea(a_{ij})),$$ we have,
	\begin{gather} \label{main-eqn1}
		\sum_{v_i \sim v_j} |x_i - a_{ij}x_j|^2 = (a-b)^2 e(V_k,V\setminus V_k) + 2ab\sum_{v_i \sim v_j,v_i\in V_k,v_j \notin V_k}(1-\Rea(a_{ij}e^{\mathbf{i}\theta})) \\ \nonumber - 2b^2\sum_{v_i \sim v_j,v_i\in V_k,v_j \notin V_k}(1-\Rea(a_{ij})) + 2b^2ma(\Phi).
	\end{gather}

	Thus, for every $k = 1,\dots,\chi$, we have
	\begin{align}
		((a^2-b^2)|V_k| + nb^2)\lambda_1(\Phi) \leq (a-b)^2 e(V_k,V\setminus V_k) + 2ab\sum_{v_i \sim v_j,v_i\in V_k,v_j \notin V_k}(1-\Rea(a_{ij}e^{\mathbf{i}\theta}))\\ \nonumber - 2b^2\sum_{v_i \sim v_j,v_i\in V_k,v_j \notin V_k}(1-\Rea(a_{ij})) + 2b^2ma(\Phi).
	\end{align}
	By adding these inequalities, we get
	\begin{align*}
		((a^2-b^2)n + n\chi b^2)\lambda_1(\Phi) &\leq (a-b)^2 (2m) + (2ab)2ma_{\theta}(\Phi) - (2b^2)2ma(\Phi) + 2b^2ma(\Phi)\chi
	\end{align*}
	which is equivalent to
	\begin{align*}
		\lambda_{1}(\Phi) &\leq \frac{2m}{n}\bigg(\frac{(a-b)^2 +(2ab)a_{\theta}(\Phi) + a(\Phi)(\chi-2)b^2)}{a^2 + b^2(\chi -1)}\bigg).
	\end{align*}
	Now, since the inequality is homogenous in $a,b$, so let $a = \gamma b$, then simplifying we obtain,
	$$
	\lambda_{1}(\Phi) \leq \frac{2m}{n}\bigg(a(\Phi) - \frac{(a(\Phi)-1)(\gamma^2+1) -2\gamma(a_{\theta}(\Phi)-1)}{\chi + \gamma^2 -1}\bigg).
	\qedhere $$
\end{proof}

By plugging $\theta =0$ in Lemma \ref{chromres-5} we obtain the following corollary.
\begin{corollary}\label{chromres-1}
	Let $\Phi = (G, \varphi)$ be a unit complex gain graph with $n$ vertices and $m$ edges and chromatic number $\chi$. Let  $V_1, V_2,\dots,V_\chi$ be the color classes of the underlying graph of $\Phi$. Then the following holds:
	\begin{align*}
		\lambda_{1}(\Phi) &\leq \frac{2m}{n}\Bigg(a(\Phi) - \frac{(\gamma-1)^2(a(\Phi)-1)}{\chi + \gamma^2-1}\Bigg),
	\end{align*}
	for any real number $ \gamma.$
\end{corollary}

\begin{theorem}\label{chromres-2}
	In the Corollary \ref{chromres-1}, the optimal bound for $\lambda_{1}(\Phi)$ is given by
	\begin{gather*}
		\lambda_{1}(\Phi) \leq \begin{cases}
			\frac{2m}{n}a(\Phi)\quad \text{if} \ 0 \leq a(\Phi) \leq 1, \\
			\frac{2m}{n}\bigg(1 - \frac{a(\Phi)-1}{\chi-1}\bigg) \quad 1 < a(\Phi) \leq 2.
		\end{cases}
	\end{gather*}
\end{theorem}
\begin{proof}
	Let us minimize the expression $\Big(a(\Phi) - \frac{(\gamma-1)^2(a(\Phi)-1)}{\chi + \gamma^2-1}\Big)$ with respect to $\gamma$. By the first order conditions we get
	\begin{align*}
		\frac{d}{d\gamma}\Bigg(a(\Phi) - \frac{(\gamma-1)^2(a(\Phi)-1)}{\chi + \gamma^2-1}\Bigg) &= 0, \\
		\implies		\frac{2(\gamma-1)}{\chi + \gamma^2-1} - \frac{2\gamma(\gamma-1)^2}{(\chi + \gamma^2-1)^2} &= 0 ,\\
		\implies		(\gamma-1)(\gamma + \chi -1) &= 0.
	\end{align*}
	Thus
	\begin{gather*}
		\gamma = \begin{cases}
			1 \quad \text{if} \ 0 \leq a(\Phi) \leq 1, \\
			-(\chi-1) \quad 1 < a(\Phi) \leq 2,
		\end{cases}
	\end{gather*}and the result follows.
\end{proof}
Note that  \cite[Theorem $2.11$]{de2011smallest} is a particular case of Theorem \ref{chromres-2}.
\begin{corollary}
	Let $G$ be a graph with $n$ vertices,  $m$ edges and chromatic number $\chi$. Then
	$$ \lambda_{1}(-) \leq \frac{2m}{n}\bigg(1-\frac{1}{\chi -1}\bigg), $$ where $\lambda_{1}(-) $ denote the least eigenvalue of the signless Laplacian.
\end{corollary}

In the Theorem \ref{chromres-2}, the bound only depends on the real part of the gain $\Phi$, whose performance is poor when $a(\Phi) = 1$, as can be seen from Table \ref{tab:bipartite-lambda-bounds}. The next theorem provides a more complete bound which uses both the real and imaginary components of the gain.


\begin{theorem}\label{chro-ima}
	In the Lemma \ref{chromres-5}, the optimal bound for $\lambda_{1}(\Phi)$ is given as
	\begin{align}\label{eq:chromres-5-opt}
		\lambda_{1}(\Phi) &\leq \frac{m}{n}\Bigg(a(\Phi) +1- \frac{a(\Phi)-1}{\chi-1} - \frac{\sqrt{\chi^2(a(\Phi)-1)^2  + 4 (\chi-1)(b(\Phi)-1)^2}}{\chi-1}\Bigg).
	\end{align}
\end{theorem}

\begin{proof}
	Let us simplify $a_{\theta}(\Phi)$ defined in Definition \ref{def:a-theta-phi} to obtain an explicit bound in terms of $\theta$, thus making it suitable for direct optimization with respect to $\gamma$.
	Let $a_{ij}=(x_{ij}+\mathbf{i}y_{ij})$, we obtain
	$$\Rea(a_{ij}e^{\mathbf{i}\theta}) = \Rea((x_{ij}+\mathbf{i}y_{ij})(\cos(\theta) +\mathbf{i}\sin(\theta))) = x_{ij}\cos\theta - y_{ij}\sin\theta. $$Substituting, we get
	\begin{align*}
		a_{\theta}(\Phi) &= \frac{1}{m}\bigg(\sum\limits_{v_i \sim v_j}(1-x_{ij}\cos\theta + y_{ij}\sin\theta)\bigg) \\ &= \frac{1}{m}\bigg(\sum\limits_{v_i \sim v_j}(1-\cos\theta + \sin\theta +(1-x_{ij})\cos\theta - \sin\theta (1-y_{ij}))\bigg)
	\end{align*}
and hence

$$		a_{\theta}(\Phi)-1 = \cos\theta(a(\Phi)-1) - \sin\theta(b(\Phi)-1).$$
Thus we have,
	\begin{align*}
		\lambda_{1}(\Phi) \leq \frac{2m}{n}\bigg(a(\Phi) - \frac{(a(\Phi)-1)(\gamma^2+1-2\gamma \cos\theta) +2\gamma \sin\theta(b(\Phi)-1)}{\chi + \gamma^2 -1}\bigg).
	\end{align*}
	Minimizing the expression with respect to $\theta$, we get
	\begin{gather*}
		(a(\Phi)-1) 2\gamma \sin\theta + (b(\Phi)-1) 2\gamma \cos\theta =0 \\
		\Rightarrow \tan\theta = -\bigg(\frac{b(\Phi)-1}{a(\Phi)-1}\bigg).
	\end{gather*}
	Let the optimal value of $\theta$ be denoted as $\theta^*$. Substituting the value of $\theta^*$ gives us
	\begin{align*}
		\lambda_{1}(\Phi) \leq \frac{2m}{n}\bigg(a(\Phi) - \frac{(a(\Phi)-1)(\gamma^2+1)-2\gamma\sqrt{(a(\Phi)-1)^2+(b(\Phi)-1)^2}}{\chi + \gamma^2 -1}\bigg).
	\end{align*}
	Now, minimizing with respect to $\gamma$, gives us
	\begin{gather*}
		\gamma^2 + (\chi-2)\cos\theta^*\gamma - (\chi-1) = 0 \\
		\Rightarrow \gamma = \frac{-(\chi-2)\cos\theta^* \pm \sqrt{(\chi-2)^2\cos^2\theta^* + 4(\chi-1)}}{2},
	\end{gather*}
	substituting the value of $\gamma$, the bound evaluates to
$$
		\lambda_{1}(\Phi) \leq \frac{m}{n}\Bigg(a(\Phi) +1- \frac{a(\Phi)-1}{\chi-1} - \frac{\sqrt{((a(\Phi)-1)^2 (\chi)^2 + 4 (b(\Phi)-1)^2 (\chi-1))}}{\chi-1}\Bigg). \qedhere
$$
\end{proof}

\begin{remark}
{\rm 	In Theorem \ref{chro-ima}, the case $\gamma = 0$ has been omitted. Because in this case the bound evaluates to
	\begin{align*}
	\lambda_{1}(\Phi) \leq \frac{2m}{n}\Bigg(a(\Phi) - \frac{a(\Phi)-1}{\chi-1}\Bigg),
	\end{align*}which is not better than the bound obtained in Theorem \ref{chromres-2}, and therefore $\gamma = 0$ doesn't correspond to the minimum solution.}
\end{remark}
The optimized bounds obtained in the Lemma \ref{chromres-1} and Theorem \ref{chromres-2}, can be optimized further for bipartite gains, which is done below, in the other theorems.

\begin{lemma}\label{chromres-3}
	Let $\Phi = (G, \varphi)$ be a bipartite unit complex gain graph with $n$ vertices and $m$ edges. Let  $V_1, V_2$ be the color classes of the underlying bipartite graph of $\Phi$. Then the following holds:
	\begin{align*}
		\	\lambda_{1}(\Phi) &\leq \frac{m}{|V_1|}\bigg(\frac{(\gamma-1)^2 + 2\gamma a_{\theta}(\Phi)}{\gamma^2 + \frac{|V_2|}{|V_1|}}\bigg)
	\end{align*}
	for all $ \gamma \in \mathbb{R}.$
\end{lemma}
\begin{proof}

	Let $\chi =2$. For $k=1$, Equation (\ref{main-eqn1}) in Lemma \ref{chromres-5} specializes as follows:
$$
		\sum_{v_i \sim v_j} |x_i - a_{ij}x_j|^2 = (a-b)^2 e(V_1, V_2) + \sum_{v_i \sim v_j,v_i\in V_1,v_j \in V_2}2ab(1-\Rea(a_{ij}e^{\mathbf{i}\theta}))
 $$
 which is equivalent to
 $$
		\sum_{v_i \sim v_j} |x_i - a_{ij}x_j|^2 = m(a-b)^2 + 2abma_{\theta}(\Phi).
$$
	
	Thus, we have
	\begin{align*}
		\lambda_{1}(\Phi) &\leq \bigg(\frac{m(a-b)^2 + 2abma_{\theta}(\Phi)}{|V_1|a^2 + |V_2|b^2}\bigg).
	\end{align*}
	Now, since the inequality is homogenous in $a,b$, and $b$ is non-zero, so let $a = \gamma b$, then simplifying we obtain,
$$
		\lambda_{1}(\Phi)  \leq \frac{m}{|V_1|}\bigg(\frac{(\gamma-1)^2 + 2\gamma a_{\theta}(\Phi)}{\gamma^2 + \frac{|V_2|}{|V_1|}}\bigg). \qedhere
$$

\end{proof}

\begin{theorem}\label{bipart-refine}
	
	{
		In the Lemma \ref{chromres-3}, the optimal bound for $\lambda_1(\Phi)$ is given by
		%
	
		{\large \begin{align}\label{eq:bipartite-opt}
		\lambda_1(\Phi) \leq
		\begin{cases}
		\frac{m}{|V_2|}, \quad \text{if} \quad a_{\theta}(\Phi) = 1,  \\
		\frac{m}{2}\Bigg(\frac{n-\sqrt{n^2-4a_{\theta}(\Phi)(2-a_{\theta}(\Phi))|V_1||V_2|}}{|V_1||V_2|}\Bigg),  \quad \text{if} \quad a_{\theta}(\Phi) \neq 1.
		\end{cases}
		\end{align}}}

\end{theorem}
\begin{proof}
	Introducing $a = a_{\theta}(\Phi)$ and $c = \frac{|V_2|}{|V_1|}$, then expanding out the inequality in Lemma \ref{chromres-3}, we have,
	\begin{align}\label{eq:lambda1-chromes-bw}
		\lambda_{1}(\Phi) &\leq \frac{m}{|V_1|}\bigg(1+ \frac{1-c+2\gamma(a-1)}{\gamma^2 + c}\bigg).
	\end{align}
	Now, differentiating the expression with respect to $\gamma$, we get,
	\begin{gather}
	\frac{2(a-1)}{\gamma^2 + c} - 2\gamma\frac{1-c+2\gamma(a-1)}{(\gamma^2  +c)^2} = 0 \label{eq:chromes-ext}\notag\\
	\Rightarrow a-1 = \gamma\frac{1-c+2\gamma(a-1)}{\gamma^2  +c}.\label{eq:chomes-quad}
	\end{gather} If $a = 1$ and $c \neq 1$ then $\gamma = 0$ from \eqref{eq:chomes-quad}, and the expression in \eqref{eq:lambda1-chromes-bw} evaluates to  $$	\lambda_{1}(\Phi) \leq \frac{m}{|V_2|}.$$ If $a = 1$ and $c = 1$ then the expression in \eqref{eq:lambda1-chromes-bw} is independent of $\gamma$ and it evaluates to  $$\lambda_{1}(\Phi) \leq \frac{m}{|V_2|}.$$Furthermore, if $a \neq 1$ then $\gamma \neq 0$, and we obtain $$\frac{a-1}{\gamma} = \frac{1-c+2\gamma(a-1)}{\gamma^2  +c}.$$Now, further simplifying we get
	\begin{gather*}
	(a-1)\gamma^2 + (1-c)\gamma - c(a-1) = 0 \\
	\Rightarrow \gamma = \frac{c-1 \pm \sqrt{(c-1)^2 + 4(a-1)^2c}}{2(a-1)}.
	\end{gather*}Thus, we obtain
	\begin{align*}
	1+ \frac{a-1}{\gamma} &= \frac{c+1 \pm \sqrt{(c-1)^2 + 4(a-1)^2c}}{2c}.
	\end{align*}
	Since we want the minimum, the positive value of $\gamma$ has to be chosen. Now, substituting back the values of $a$ and $c$ and using $n = |V_1| + |V_2|$, we obtain
	\begin{align*}
	\gamma &= \Bigg(\frac{|V_2|-|V_1|+\sqrt{n^2-4a_{\theta}(\Phi)(2-a_{\theta}(\Phi))|V_1||V_2|}}{2(a_{\theta}(\Phi)-1)|V_1|}\Bigg),
	\end{align*}
	and correspondingly those evaluates to
	$$
	\lambda_1(\Phi) \leq \frac{m}{2}\Bigg(\frac{n-\sqrt{n^2-4a_{\theta}(\Phi)(2-a_{\theta}(\Phi))|V_1||V_2|}}{|V_1||V_2|}\Bigg). \qedhere
	$$
\end{proof}



Next we prove some other bounds for $\lambda_{1}(\Phi)$ in terms of degrees. These bounds extend the bounds established in  \cite[Theorem 2.7, Theorem 2.8 and Theorem 2.9]{de2011smallest}.

\begin{theorem}\label{theorem:d_s+d_t-bound}
	Let $\Phi = (G, \varphi)$ be a connected nonempty gain graph with $n$ vertices and $d_s$ denote the degree of the vertex $v_s \in V(G)$. Then the following statements hold:
	\begin{enumerate}
		\item[(i)]  $\lambda_1(\Phi) \leq \min_{v_s \sim v_t} \bigg(\frac{1}{2}\{d_s+d_t-2\}\bigg)$.
		
		
		\item[(ii)] $\lambda_1(\Phi) \leq \min_{v_s \sim v_t}\frac{1}{2}\Bigg(d_s+d_t- \sqrt{(d_s-d_t)^2+4}\Bigg)$.
		
		\item[(iii)] If $\delta >0$ denotes the minimum degree of $G$, then $$\lambda_1(\Phi) \leq\frac{1}{2}\Bigg(\delta+n-1- \sqrt{(n-1-\delta)^2+4}\Bigg) < \delta.$$
	\end{enumerate}
\end{theorem}

\begin{proof}
	\begin{enumerate}
		
		\item[(i)] Let $v_s \sim v_t$ with $s<t$. Define the vector $\mathbf{x} \in \mathbb{C}^n$ as follows: $$x_k = \begin{cases} \frac{1}{\sqrt{2}} & ~\mbox{if}~k=s, \\ \frac{1}{\sqrt{2}}\overline{\varphi_{st}} & ~\mbox{if}~k=t, \\  0 & \text{otherwise.} \end{cases}$$ Since $\Vert \mathbf{x} \Vert=1$, by Theorem \ref{theorem:CF}, we have $\lambda_1(\Phi) \leq \sum_{v_i \sim v_j}|x_i - \varphi_{ij}x_j|^2 = \frac{1}{2}(d_s + d_t - 2).$

		\item[(ii)] Let $v_s \sim v_t$ with $s<t$ and $d_s \geq d_t$.  Define the vector $\mathbf{x} \in \mathbb{C}^n$ as follows: $$x_k = \begin{cases} a & ~\mbox{if}~k=s, \\ \overline{\varphi_{st}} \sqrt{1-a^2} & ~\mbox{if}~ k= t, \\  0 & \text{otherwise,} \end{cases}$$ where $|a|\leq 1$. By Theorem \ref{theorem:CF}
		\begin{align*}
			\lambda_1(\Phi) &\leq \sum_{v_i \sim v_j}|x_i - \varphi_{ij}x_j|^2\\ &= a^2(d_s-1) + (1-a^2)(d_t-1) + (a-\sqrt{1-a^2})^2\\
			&\leq \min_{|a|\leq 1} \{a^2(d_s-1) + (1-a^2)(d_t-1) + 1 - 2a\sqrt{1-a^2}\} \notag \\
			&= \min_{|a|\leq 1} \{d_t + a^2(d_s-d_t) - 2a\sqrt{1-a^2}\}. \label{opt:a}
		\end{align*}Differentiating the above expression  with respect to $a$, and equating the obtained expression to $0$, we obtain
		\begin{align*}
			a^4[(d_s-d_t)^2+4] - a^2[(d_s-d_t)^2+4]+1 = 0,
		\end{align*}which after solving yields $$a =  \sqrt{\frac{1}{2}-\frac{d_s-d_t}{2\sqrt{(d_s-d_t)^2+4}}},$$for the minimum. Thus, we obtain $$\lambda_1(\Phi) \leq \min_{v_s \sim v_t} \frac{1}{2}\Bigg(d_s+d_t- \sqrt{(d_s-d_t)^2+4}\Bigg).$$ Proof of the case $d_s < d_t$ is left to the reader.

		\item[(iii)] Let $d_s = \delta$ and let $v_t$ be a neighbor of $v_s$. By statement $(ii)$, we see that $$\lambda_1(\Phi) \leq \min_{v_s \sim v_t}\frac{1}{2}\Bigg(\delta+d_t- \sqrt{(d_t-\delta)^2+4}\Bigg).$$It is clear that the function$$f(x) =\frac{1}{2}\Bigg(\delta+x- \sqrt{(x-\delta)^2+4}\Bigg)$$is increasing in $x$ for $x\geq \delta$. Therefore, in our case $f(x) \leq f(n - 1)$, implying the assertion. \qedhere
	\end{enumerate}
\end{proof}

\begin{remark}{\rm		In this setup, the bound given in statement $(ii)$ of Theorem \ref{theorem:d_s+d_t-bound} is optimal.}
	
\end{remark}

In the following theorem, we obtain  optimal upper bounds for $\lambda_1(\Phi)$.
\begin{theorem}\label{theorem:d_s+d_t+d_i-bound}
	Let $\Phi = (G, \varphi)$ be a connected gain graph with $n$ vertices. Let $d_s$ denote the degree of the vertex $v_s \in V(G)$ and let $\mathcal{T}_G$ be the set of triples $(i, j,k)$ such that the vertices $v_i, v_j, v_k$ form a triangle $T_{i,j,k}$ in $G$. Assuming that $\mathcal{T}_G$ is nonempty and setting $\cos \theta_{ijk} = \Rea(\varphi(T_{i, j,k}))$ for each $(i,j,k) \in \mathcal{T}_G$, the following inequalities hold:
	{\small
		\begin{enumerate}
			\item[(i)] \begin{eqnarray}\label{eq:3d-degree-bound1-triangle}
				\lambda_1(\Phi) \leq \min_{(s,t,r) \in \mathcal{T}_G} \: \frac{d_s+d_t+d_r-2\cos\theta_{str}-4}{3}.
			\end{eqnarray}
			\item[(ii)] \begin{eqnarray}\label{eq:3d-degree-bound2-triangle}
				\lambda_1(\Phi) \leq \min_{(s,t,r) \in \mathcal{T}_G} \frac{1}{4}\bigg(d_s+d_t+2d_r-2-\sqrt{(d_s+d_t-2d_r-2)^2+8(\cos\theta_{str}+1)^2}\bigg).
			\end{eqnarray}
			
			\item[(iii)] \begin{eqnarray}\label{eq:3d-degree-bound3-triangle}
				\lambda_1(\Phi) \leq \min_{(s,t,r) \in \mathcal{T}_G} \frac{1}{4}\bigg(d_s+2d_t+d_r-2-\sqrt{(d_s+d_r-2d_t-2)^2+8(\cos\theta_{str}+1)^2}\bigg).
			\end{eqnarray}

			\item[(iv)]\begin{eqnarray}\label{eq:3d-degree-bound4-triangle}
				\lambda_1(\Phi) \leq \min_{(s,t,r) \in \mathcal{T}_G} \frac{1}{4}\bigg(2d_s+d_t+d_r-2\cos\theta_{str}-\sqrt{(2d_s-d_t-d_r-2\cos\theta_{str})^2+32}\bigg).
			\end{eqnarray}
			
	\end{enumerate}}
\end{theorem}

\begin{proof}
	Let the vertices $v_s, v_t, v_r$ form a triangle in $G$ and $s<t<r$. Define the vector $\mathbf{x} \in \mathbb{C}^n$ as follows: $$x_k = \begin{cases} a &~\mbox{if}~ k=s, \\ b\overline{\varphi}_{st} &~\mbox{if}~ k= t, \\  c\overline{\varphi}_{sr} &~\mbox{if}~ k= r, \\ 0 & \text{otherwise,} \end{cases}$$ where $a^2+b^2+c^2=1$. By Theorem \ref{theorem:CF}, we obtain

\begin{align}
\lambda_1(\Phi) &\leq \sum_{v_i \sim v_j}|x_i - \varphi_{ij}x_j|^2 \notag\\
&=\sum_{\substack{v_j: v_s \sim v_j;\\ v_j \neq v_t; v_j \neq v_r}}a^2  + \sum_{\substack{v_j: v_t \sim v_j;\\ v_j \neq v_s; v_j \neq v_r}}b^2  + \sum_{\substack{v_j: v_r \sim v_j;\\ v_j \neq v_s; v_j \neq v_t}}c^2  + |a - \overline{\varphi}_{st}\varphi_{st}b|^2 + |a - \overline{\varphi}_{sr}\varphi_{sr} c|^2 \notag \\
&\quad \quad+ \quad |b - \varphi_{st}\varphi_{tr} \varphi_{rs}c|^2. \notag\\
&= d_sa^2+ d_tb^2+ d_rc^2 - 2ab - 2ac - bc (\varphi_{st}\varphi_{tr} \varphi_{rs} + \overline{\varphi}_{st}\overline{\varphi}_{tr} \overline{\varphi}_{rs}).
\end{align}

Thus,  we  obtain the following:
\begin{align}\label{eq:traingle-general-opt}
\lambda_1(\Phi) &\leq \min_{(s,t,r) \in \mathcal{T}_G}\min_{a^2+b^2+c^2=1} \: \{a^2(d_s-d_r) + b^2(d_t-d_r) - 2(ab+bc \cos\theta_{str}+ac)+d_r\}
\end{align}
	
	\begin{enumerate}
		\item[(i)] Choosing $a = b = c = \frac{1}{\sqrt{3}}$, we get:
		$$\lambda_1(\Phi) \leq \min_{(s,t,r) \in \mathcal{T}_G} \: \frac{d_s+d_t+d_r-2\cos\theta_{str}-4}{3}.$$

		\item[(ii)] Let us take $a = b$, and introduce $\alpha = d_s+d_t-2d_r-2 ,\: \beta= \cos\theta_{str}+1$. The inner optimization problem in \eqref{eq:traingle-general-opt} becomes
		\begin{eqnarray}\label{eq:traingle-a=b}
			\min_{|a| \leq 1} \: \{a^2\alpha-2a\sqrt{1-2a^2}\beta\}  + d_r.
		\end{eqnarray}After differentiating the above expression w.r.t $a$, we obtain
		\begin{align*}
			a \alpha-\sqrt{1-2a^2}\beta+\frac{2a^2}{\sqrt{1-2a^2}} &=0 \\
			\Rightarrow a^4(2\alpha^2+16\beta^2) - a^2 (\alpha^2+8\beta^2) + \beta^2 &= 0,
		\end{align*}
		which after solving yields $$a^2 = \frac{1}{4}\Bigg(1-\frac{\alpha}{\sqrt{\alpha^2+8\beta^2}}\Bigg),$$ for the minimum. Thus, after putting this value of $a$ in \eqref{eq:traingle-a=b}, we obtain
		\begin{align*}
			\lambda_1(\Phi) \leq \min_{(s,t,r) \in \mathcal{T}_G} \frac{1}{4}\bigg(d_s+d_t+2d_r-2-\sqrt{(d_s+d_t-2d_r-2)^2+8(\cos\theta_{str}+1)^2}\bigg).
		\end{align*}
		When  $\alpha =\beta = 0$, then using \eqref{eq:traingle-a=b}, we obtain the upper bound as, $\lambda_1(\Phi) \leq \min\limits_{(s,t,r) \in \mathcal{T}_G} d_r$. That is, the bound reduces to the smallest degree of vertices forming the triangle.
		\item[(iii)] Let us take $a = c$, and introduce $\alpha = d_s-2d_t+d_r-2 ,\: \beta= \cos\theta_{str}+1$. The inner optimization problem in \eqref{eq:traingle-general-opt} becomes
		\begin{eqnarray}\label{eq:traingle-a=c}
			\min_{|a| \leq 1} \: \{a^2\alpha-2a\sqrt{1-2a^2}\beta\}  + d_t,
		\end{eqnarray}which yields $$a^2 = \frac{1}{4}\Bigg(1-\frac{\alpha}{\sqrt{\alpha^2+8\beta^2}}\Bigg)$$ for the minimum. Thus, after putting the value of $a$ in \eqref{eq:traingle-a=c}, we obtain
		\begin{align*}
			\lambda_1(\Phi) \leq \min_{(s,t,r) \in \mathcal{T}_G} \frac{1}{4}\bigg(d_s+2d_t+d_r-2-\sqrt{(d_s+d_r-2d_t-2)^2+8(\cos\theta_{str}+1)^2}\bigg).
		\end{align*}

		\item[(iv)] Finally, let us take $b = c$, and introduce $\alpha = d_t+d_r-2d_s-2\cos\theta_{str} ,\: \beta= 2$. The inner optimization problem in \eqref{eq:traingle-general-opt} becomes
		\begin{eqnarray}\label{eq:traingle-b=c}
			\min_{|b| \leq 1} \: \{b^2\alpha-2b\sqrt{1-2b^2}\beta\}  + d_s
		\end{eqnarray}which yields $$b^2 = \frac{1}{4}\Bigg(1-\frac{\alpha}{\sqrt{\alpha^2+8\beta^2}}\Bigg)=\frac{1}{4}\Bigg(1-\frac{\alpha}{\sqrt{\alpha^2+32}}\Bigg)$$ for the minimum. Thus, after putting this value of $b$ in \eqref{eq:traingle-b=c}, we obtain
$$
			\lambda_1(\Phi) \leq \min_{(s,t,r) \in \mathcal{T}_G} \frac{1}{4}\bigg(2d_s+d_t+d_r-2\cos\theta_{str}-\sqrt{(2d_s-d_t-d_r-2\cos\theta_{str})^2+32}\bigg). \qedhere
$$
		
	\end{enumerate}
	
\end{proof}


\begin{remark}{\rm
		If the vertices $v_s, v_t$ and $v_r$ are such that  $v_s\sim v_t$ , $v_s \sim v_r$ and $v_t, v_r$ are not adjacent, then we obtain the bound as
		\begin{align*}
			\lambda_1(\Phi) &\leq \min_{v_t \sim v_s \sim v_r}\min_{a^2+b^2+c^2=1} \: \sum_{v_i \sim v_j}|x_i - \varphi_{ij}x_j|^2\\
			&= \min_{v_t \sim v_s \sim v_r}\min_{a^2+b^2+c^2=1} \: \{a^2(d_s-d_r) + b^2(d_t-d_r) - 2(ab+ac)+d_r\}.
		\end{align*} Note that this case is equivalent to choosing $\cos\theta_{str}=0$ in Theorem \ref{theorem:d_s+d_t+d_i-bound}. Therefore, after substituting $\cos\theta_{str}=0$ for the bounds obtained in Theorem \ref{theorem:d_s+d_t+d_i-bound}, we get the corresponding bounds as follows.
		\begin{enumerate}
			\item For $a = b = c = \frac{1}{\sqrt{3}}$
			\begin{eqnarray}\label{eq:3d-degree-bound1-tree}
				\lambda_1(\Phi) \leq \min_{v_t \sim v_s \sim v_r} \: \frac{d_s+d_t+d_r-4}{3}.
			\end{eqnarray}
			
			\item For $a = b$
			\begin{eqnarray}\label{eq:3d-degree-bound2-tree}
				\lambda_1(\Phi) \leq \min_{v_t \sim v_s \sim v_r} \frac{1}{4}\bigg(d_s+d_t+2d_r-2-\sqrt{(d_s+d_t-2d_r-2)^2+8}\bigg).
			\end{eqnarray}
			
			\item For $a = c$
			\begin{eqnarray}\label{eq:3d-degree-bound3-tree}
				\lambda_1(\Phi) \leq \min_{v_t \sim v_s \sim v_r} \frac{1}{4}\bigg(d_s+2d_t+d_r-2-\sqrt{(d_s-2d_t+d_r-2)^2+8}\bigg).
			\end{eqnarray}
			
			\item For $b = c$
			\begin{eqnarray}\label{eq:3d-degree-bound4-tree}
				\lambda_1(\Phi) \leq \min_{v_t \sim v_s \sim v_r} \frac{1}{4}\bigg(2d_s+d_t+d_r-\sqrt{(2d_s-d_t-d_r)^2+32}\bigg).
			\end{eqnarray}
			
		\end{enumerate}
	}
\end{remark}

\section{Bounds for the largest eigenvalue of gain Laplacian}\label{sec:largest-eig-bounds}
The aim of this section is to establish bounds for the largest eigenvalues of the gain Laplacian matrices. It  is known that  $\lambda_n(\Phi)  \geq  \Delta + 1$ \cite[Theorem 5.8]{reff2012spectral}. In the next theorem, we establish an improvement of this bound for certain values of $k$.


\begin{lemma}\label{lemma:diagonal-L-k}
	Let $\Phi  = (G,\varphi)$ be a gain graph. Let $\mathbf{L}(\Phi)$ be the Laplacian matrix of the gain graph. Then  $\lambda_n(\Phi) \geq \bigg(\max\limits_{i \in \{1,2,\dots,n\}}\big(\mathbf{L}^k(\Phi)_{ii}\big)\bigg)^{\frac{1}{k}}.$

\end{lemma}
\begin{proof}
	By taking the $i$-th diagonal entry of $\mathbf{L}^k(\Phi)$ in the Corollary \ref{corollary:max-diag-element}, we get the result.
\end{proof}

In Table $3$, we compare the above bound with the known bounds for $\lambda_n(\Phi)$. We infer that, for the chosen graphs and for $k = 6, 7, 9, 11, 12$ the above bound is slightly better than $\Delta +1.$

Next we show that the bound in the previous theorem converge to $\lambda_n(\Phi)$, and hence giving an iterative scheme to approximate $\lambda_n(\Phi)$. This is done in the next theorem.
\begin{theorem}\label{thm:diag-elem}
	Let $\Phi  = (G,\varphi)  $ be a gain graph. Let $\mathbf{L}(\Phi)$ be the Laplacian matrix of the gain graph and let $L_{ii}$ denote the $i^{th}$ diagonal entry of the matrix $L$. Then $$ \lambda_n(\Phi) = \lim_{k \to \infty}\bigg(\max\limits_{i \in \{1,2,\dots,n\}} \big(\mathbf{L}^k(\Phi)_{ii}\big)\bigg)^{\frac{1}{k}}. $$
\end{theorem}
\begin{proof} Let $\lambda_1(\Phi) \leq \dots \leq \lambda_{p-1}(\Phi) < \lambda_p(\Phi) = \dots = \lambda_n(\Phi)$ be the eigenvalues of $\mathbf{L}(\Phi).$
	By the
	spectral decomposition of $\mathbf{L}(\Phi)$, we have
	$$\mathbf{L}(\Phi) = \mathbf{U}\mathbf{\Lambda}(\Phi)\mathbf{U}^*= 	\begin{bmatrix}
		u_{11} & u_{21} & \dots & u_{n1} \\
		u_{12} & u_{22} & \dots & u_{2n}\\
		\vdots  & \vdots & \ddots & \vdots \\
		u_{1n} & u_{2n} & \dots & u_{nn}
	\end{bmatrix}
	\begin{bmatrix}
		\lambda_1(\Phi) & 0 & \dots & 0 \\
		0 & \lambda_2(\Phi) & \dots & 0\\
		\vdots  & \vdots & \ddots & \vdots \\
		0 & 0 & \dots & \lambda_n(\Phi)
	\end{bmatrix}\begin{bmatrix}
		u_{11}^* & u_{12}^* & \dots & u_{1n}^* \\
		u_{21}^* & u_{22}^* & \dots & u_{2n}^*\\
		\vdots  & \vdots & \ddots & \vdots \\
		u_{n1} ^*& u_{n2}^* & \dots & u_{nn}^*
	\end{bmatrix}.$$
	Thus $$\mathbf{L}^k(\Phi)_{ii} = \sum_{j=1}^n\lambda_j^k(\Phi) \: |u_{ji}|^2.$$ Now
	\begin{align*}
		\Bigg(\sum_{j=1}^n\lambda_j(\Phi)^k \: |u_{ji}|^2\Bigg)^{\frac{1}{k}}  = \lambda_n(\Phi) \Bigg( \sum_{j=1}^n\Bigg(\frac{\lambda_j(\Phi)}{\lambda_n(\Phi)}\Bigg)^k \: |u_{ji}|^2\Bigg)^{\frac{1}{k}},
	\end{align*}so when $k \to \infty$, $\bigg(\frac{\lambda_j(\Phi)}{\lambda_n(\Phi)}\bigg)^k \to 0$. We obtain,
	
\textbf{	\begin{align*}
		\lim_{k \to \infty} &\lambda_n(\Phi) \Bigg( \max\limits_{i \in \{1,2,\dots,n\}} \: \sum_{j=1}^n\Bigg(\frac{\lambda_j(\Phi)}{\lambda_n(\Phi)}\Bigg)^k \: |u_{ji}|^2\Bigg)^{\frac{1}{k}} \\
		&=  \lambda_n(\Phi) \Bigg( \max\limits_{i \in \{1,2,\dots,n\}} \: \lim_{k \to \infty} \Bigg[ \sum_{j=1}^n\Bigg(\frac{\lambda_j(\Phi)}{\lambda_n(\Phi)}\Bigg)^k \: |u_{ji}|^2\Bigg]\Bigg)^{\frac{1}{k}}  \\
		&= \lambda_n(\Phi) \: \lim_{k \to \infty} \Bigg( \max\limits_{i \in \{1,2,\dots,n\}} \: \sum_{j=p}^n |u_{ji}|^2\Bigg)^{\frac{1}{k}}\\ &= \lambda_n(\Phi). 	
	\end{align*}}
	Note that $U$ is unitary, so $  \max\limits_{i \in \{1,2,\dots,n\}}  \: \sum_{j=p}^n |u_{ji}|^2$ is positive.
\end{proof}
 Using the above bound, we can derive a lower bound for the smallest eigenvalue of an \emph{unbalanced} gain graph.
\begin{corollary}
	Let $\Phi$ be a unbalanced gain graph. Let $\mathbf{L}(\Phi)$ be the Laplacian matrix of the gain graph and let $L_{ii}$ denote the $i^{th}$ diagonal entry of the matrix $L$. Then $$ \lambda_1(\Phi) = \lim_{k \to \infty}\bigg(\max\limits_{i \in \{1,2,\dots,n\}} \big(\mathbf{L}^{-k}(\Phi)_{ii}\big)\bigg)^{-\frac{1}{k}} .$$
\end{corollary}
\begin{proof}
	Since $\Phi$ is unbalanced, $\mathbf{L}^{-1}(\Phi)$ exists and its maximum eigenvalue is $\frac{1}{\lambda_{1}(\Phi)}$. By applying the Theorem \ref{thm:diag-elem} to $\mathbf{L}^{-1}(\Phi)$, we have  desired result.
\end{proof}

The following result will be helpful in the proof of Theorem \ref{theorem:avg-2-degree}.
\begin{lemma}[{\cite[Theorem 4.4]{reff2012spectral}}]\label{lemma:phi-psi}
	Let $\Phi = (G,\varphi)$ be a unit complex gain graph on $n$ vertices. Let $\lambda_n(\Phi)$ denote the largest eigenvalue of the Laplacian $\mathbf{L}(\Phi)$. Then $$\lambda_n(\Phi) \leq \lambda_n(-),$$where $\lambda_{n}(-) $ denote the largest eigenvalue of the signless Laplacian of $G$.  Equality holds if and only if $\Phi \sim (G, -).$
\end{lemma}
In the next theorem we establish an upper  bound for the  largest eigenvalue of the gain Laplacian matrix using arbitrary invertible diagonal matrices.  For some particular choices of the diagonal matrix, we obtain parametric bounds for the largest eigenvalue of the gain Laplacian (\textit{cf.} Remark \ref{rem-inv-mat}), which can be tuned to achieve superior performance as compared to the existing bounds.
\begin{theorem}\label{theorem:avg-2-degree}
	Let $\Phi = (G,\varphi)$ be a connected $\mathbb{T}$-gain graph on $n$ vertices. Let $\mathbf{C}$ be an invertible diagonal matrix given as $\diag(c_1,c_2,\dots,c_n) \in \mathbb{C}^{n\times n}$. Then, $$\lambda_n(\Phi) \leq \max\limits_{v_i \in V(\Phi)}\bigg\{d_i+\sum_{v_i\sim v_j}\frac{|c_j|}{|c_i|}\bigg\}.$$	
\end{theorem}
\begin{proof}

	Consider the matrix $\mathbf{H}(\Phi) = \mathbf{C}^{-1}\mathbf{L}(\Phi)\mathbf{C}$. Then
	\begin{gather*}
		\mathbf{H}(\Phi) = \begin{bmatrix}
			d_1 & \frac{c_2}{c_1}\mathbf{L}(\Phi)_{12} & \dots & \frac{c_n}{c_1}\mathbf{L}(\Phi)_{1n} \\
			\frac{c_1}{c_2}\mathbf{L}(\Phi)_{21} & d_2 & \dots & \frac{c_n}{c_2}\mathbf{L}(\Phi)_{2n}\\
			\vdots  & \vdots & \ddots & \vdots \\
			\frac{c_1}{c_n}\mathbf{L}(\Phi)_{n1} & \frac{c_2}{c_n}\mathbf{L}(\Phi)_{n2} & \dots & d_n
		\end{bmatrix}.
	\end{gather*}
	Now, by Gershgorin's circle theorem,  we have,
	\begin{gather}\label{gen-2-deg}
		\lambda_n(\Phi) \leq \max\limits_{v_i \in V(\Phi)}\bigg\{d_i+\sum_{v_i\sim v_j}\frac{|c_j|}{|c_i|}\bigg\}.
	\end{gather}	
	\textbf{Equality:} Since the above proof technique is independent of the choice of gain, we obtain $\lambda_n(-) \leq \max\limits_{v_i \in V(\Phi)}\{d_i+\sum_{v_i\sim v_j}\frac{|c_j|}{|c_i|}\}.$ By Lemma \ref{lemma:phi-psi}, we have $\lambda_n(\Phi) \leq \lambda_n(-) \leq \max\limits_{v_i \in V(\Phi)}\{d_i+\sum_{v_i\sim v_j}\frac{|c_j|}{|c_i|}\}.$ Thus, in case of equality it is necessary that $\lambda_n(\Phi)  = \lambda_n(-)$ holds, and which we know, by Lemma \ref{lemma:phi-psi}, holds if and only if $(G,\phi) \sim (G, -)$.
\end{proof}
Below we provide some concrete bounds by considering suitable values for $\mathbf{C}$. For a vertex $v$ of $V(\Phi)$, we denote the average $2$-degree by $m_i = \frac{\sum_{v_i \sim v_j} d_j}{d_i}$ and different types of generalized $2$-degrees as follows: \begin{enumerate}
	\item $m_i^k = \frac{\sum_{v_i \sim v_j} (d_j+ m_j^{k-1})}{(d_i + m_i^{k-1})}$ with the convention that, $m_j^{0} = 0$ and $k = 1,2, \dots$ for all vertices $v_j \in V(\Phi)$.
	\item $n_i^k = \frac{\sum_{v_i \sim v_j} (d_j+ n_j^{k-1} -r)}{(d_i + n_i^{k-1}-r)}$ with the convention that, $n_j^{0} = 0$ and $k = 1,2, \dots$ for all vertices $v_j \in V(\Phi)$, and $0< r <1$. 
	\item $l_i^k = \frac{\sum_{v_i \sim v_j} (d_j+ l_j^{k-1}- l_j^{k-2})}{(d_i + l_i^{k-1}-l_i^{k-2})}$ with the convention that, $l_j^{0} = 0$, $l_j^{1} = m_j$ and $k = 2,3, \dots$ for all vertices $v_j \in V(\Phi)$.
\end{enumerate}
Note that the choices of initial conditions for the recurrence relations given above are arbitrary.

\begin{remark}\label{rem-inv-mat}{\rm
		In  Theorem \ref{theorem:avg-2-degree}, the diagonal matrix $C$ can be chosen to give some concrete bounds as follows:
		\begin{enumerate}
			\item If $\mathbf{C} = \diag(d_1, \dots, d_n)$, then we have,
			\begin{align}\label{eq: gershgorian-m}
				\lambda_n(\Phi) \leq \max_{v_i \in V(\Phi)} \{d_i  + m_i\}.
			\end{align}
			\item If $\mathbf{C} = \begin{bmatrix} (d_1 + m_1^{k-1})\mathbf{e}_1 & (d_2 + m_2^{k-1})\mathbf{e}_2 & \dots &  (d_n + m_n^{k-1})\mathbf{e}_n \end{bmatrix}$, then we have,
			\begin{align}\label{eq: gershgorian-mk}
				\lambda_n(\Phi) \leq \max_{v_i \in V(\Phi)} \{d_i  + m_i^k\}.
			\end{align}
			\item  If $\mathbf{C} = \begin{bmatrix} (d_1 + n_1^{k-1} - r)\mathbf{e}_1 & (d_2 + n_2^{k-1}-r)\mathbf{e}_2 & \dots &  (d_n + n_n^{k-1} - r)\mathbf{e}_n \end{bmatrix}$, then we have,
			\begin{align}\label{eq: gershgorian-n}
				\lambda_n(\Phi) \leq \max_{v_i \in V(\Phi)} \{d_i  + n_i^k\}.
			\end{align}
			\item If $\mathbf{C} = \begin{bmatrix} (d_1 + l_1^{k-1} - l_1^{k-2})\mathbf{e}_1 & (d_2 + l_2^{k-1} - l_2^{k-2})\mathbf{e}_2 & \dots &  (d_n + l_n^{k-1} - l_n^{k-2})\mathbf{e}_n \end{bmatrix}$, then we have,
			\begin{align}\label{eq: gershgorian-l}
				\lambda_n(\Phi) \leq \max_{v_i \in V(\Phi)} \{d_i  + l_i^k\}.
			\end{align}
	\end{enumerate}}	
\end{remark}

In Table \ref{tab:bounds-for-max-eig}, we provide a comparison of the above bounds with those existing in the literature. We observe that some of the above bounds are better than  known bounds in the literature.

Next, we provide a sequence of bounds for $\lambda_{n}(\Phi)$ in terms of traces of powers of Laplacian, and prove that the sequence converges to $\lambda_{n}(\Phi)$.
\begin{theorem}\label{trace-bound}
	Let $\Phi = (G, \varphi)$ be a complex unit gain graph with $n$ vertices with $n \geq 2$. Let $\Tr(L)$ denote the trace of the matrix $L$. Then
	\begin{align}\label{eq:trace-powers-lambda}
		\lambda_n(\Phi) &\geq \sqrt[k]{\frac{\Tr(L^{k}(\Phi))}{n} +  \sqrt{\frac{1}{n^2(n-1)}\bigg(n\Tr(L^{2k}(\Phi)) - \Tr(L^k(\Phi))^2\bigg)}} \quad \text{and}\\
		\lambda_n(\Phi) &= \lim_{k \to \infty}
		\sqrt[k]{\frac{\Tr(L^{k}(\Phi))}{n} +  \sqrt{\frac{1}{n^2(n-1)}\bigg(n\Tr(L^{2k}(\Phi)) - \Tr(L^k(\Phi))^2\bigg)}}. \notag
	\end{align}
\end{theorem}
\begin{proof}
	 Let $\lambda_1(\Phi) \leq \dots \leq \lambda_{p-1}(\Phi) < \lambda_p(\Phi) = \dots = \lambda_n(\Phi)$ be the eigenvalues of $\mathbf{L}(\Phi).$
	It is clear that,
	\begin{gather}\label{eq:trace-lambda-n}
		( n\lambda_n^k - \Tr(L^{k}(\Phi)))^2 = \Bigg(\sum_{i=1}^{n} (\lambda_n^k - \lambda_i^k)\Bigg)^2 \geq \sum_{i=1}^{n} (\lambda_n^k - \lambda_i^k)^2.
	\end{gather}
	Now,
	$$
	\sum_{i=1}^{n} (\lambda_n^k - \lambda_i^k)^2 = \sum_{i=1}^{n} \big(\lambda_n^{2k} + \lambda_i^{2k} -2\lambda_n^{k}\lambda_i^{k}\big) = \Tr(L^{2k}(\Phi)) + n\lambda_n^{2k} - 2\lambda_n^k\Tr(L^k(\Phi)).$$ Also $$
	\sum_{i=1}^{n} (\lambda_n^k - \lambda_i^k)^2 = \frac{1}{n}( n\lambda_n^k - \Tr(L^{k}(\Phi)))^2 + \Tr(L^{2k}(\Phi)) - \frac{\Tr(L^k(\Phi))^2}{n}.
	$$
	Thus
	$$( n\lambda_n^k - \Tr(L^{k}(\Phi)))^2 \geq \frac{1}{n}( n\lambda_n^k - \Tr(L^{k}(\Phi)))^2 + \Tr(L^{2k}(\Phi)) - \frac{\Tr(L^k(\Phi))^2}{n}, $$	
	and hence	
	$$( n\lambda_n^k - \Tr(L^{k}(\Phi)))^2 \geq \frac{1}{n-1}\bigg(n\Tr(L^{2k}(\Phi)) - \Tr(L^k(\Phi))^2\bigg).$$ From the above inequality, it is easy to see that$$
	\lambda_n(\Phi) \geq \sqrt[k]{\frac{\Tr(L^{k}(\Phi))}{n} +  \sqrt{\frac{1}{n^2(n-1)}\bigg(n\Tr(L^{2k}(\Phi)) - \Tr(L^k(\Phi))^2\bigg)}}.$$Now, since in the limit of $k \to \infty$, $({\lambda_j/\lambda_n})^k \to 0$, we obtain
	\begin{align*}
		&\lim_{k \to \infty}
		\sqrt[k]{\frac{\Tr(L^{k}(\Phi))}{n} +  \sqrt{\frac{1}{n^2(n-1)}\bigg(n\Tr(L^{2k}(\Phi)) - \Tr(L^k(\Phi))^2\bigg)}} \\
		&= \lim_{k \to \infty}
		\sqrt[k]{\frac{\sum_{i=1}^n\lambda_i^k}{n} +  \sqrt{\frac{1}{n^2(n-1)}\bigg(n\sum_{i=1}^n\lambda_i^{2k} - (\sum_{i=1}^n\lambda_i^k)^2\bigg)}} \\
		&= \lambda_n \lim_{k \to \infty}
		\sqrt[k]{\frac{1}{n}\sum_{i=1}^n\Bigg(\frac{\lambda_i}{\lambda_n}\Bigg)^k +  \sqrt{\frac{1}{n^2(n-1)}\Bigg[n\sum_{i=1}^n\Bigg(\frac{\lambda_i}{\lambda_n}\Bigg)^{2k} - \Bigg(\sum_{i=1}^n\Bigg(\frac{\lambda_i}{\lambda_n}\Bigg)^k\Bigg)^2\Bigg]}} \\
		&= \lambda_n \lim_{k \to \infty}
		\sqrt[k]{\frac{n-p+1}{n} +  \sqrt{\frac{1}{n^2(n-1)}\bigg[n(n-p+1) - (n-p+1)^2\bigg]}} \\
		&= \lambda_n \lim_{k \to \infty}\Bigg(\frac{n-p+1}{n} +  \sqrt{\frac{(n-p+1)(p-1)}{n^2(n-1)}}\:\Bigg)^{\frac{1}{k}} = \lambda_n(\Phi).
	\end{align*}
	\textbf{Equality:} We know that equality in \eqref{eq:trace-lambda-n} follows if and only if $\lambda_i(\Phi) = \lambda_n(\Phi)$, for $i=1,2, \ldots, n-1$. This condition is satisfied for $\Phi \sim (K_n,1)$.
\end{proof}


\section{ Comparison between extremal eigenvalues bounds}\label{sec:comparison}

In this section, we perform a comparative analysis for different bounds obtained in this paper for both $\lambda_1(\Phi)$ and $\lambda_n(\Phi)$.

We perform all of our experiments\footnote{Code available at: \url{https://github.com/KumarNavish/gain-extremal-bounds}} on a set of Erd\H{o}s-Re\'yni graphs, $G(n,p)$, which are constructed by connecting nodes randomly. Each edge is included in the graph with probability $p$ independent from every other edge. Equivalently, all graphs with $n$ nodes and $m$ edges have equal probability  $p^m(1-p)^{\binom{n}{2}- m}$.

We used \texttt{Networkx} package to generate these random graphs. The way we define random complex gains of unit modulus is as follows:
\begin{itemize}
	\item First, we generate a random $n \times n$ matrix containing unit modulus complex numbers using \texttt{Numpy} (seed is set to 0). Then we take Hadamard product of this matrix with the
	$(0,1)$-adjacency matrix.
	\item Finally, we add the matrix obtained in last step with its Hermitian conjugate transpose and divide each entry by its modulus to obtain the gain adjacency matrix.
\end{itemize}

We construct a signed bipartite complete graph $K_{5,15}$, whose  adjacency matrix is given by
\begin{align}\label{eq: bipartite-adj}
	B = \begin{bmatrix}
		0   & A\\
		A^T & 0
	\end{bmatrix},
\end{align}where
\begin{align*}
	\setcounter{MaxMatrixCols}{15}
	A = \begin{bmatrix}
		1 &  -1 & \phantom{-}1 & \phantom{-}1 & \phantom{-}1 & \phantom{-}1 & -1 & -1 & -1 & -1 & \phantom{-}1 & \phantom{-}1 & \phantom{-}1 & \phantom{-}1 & \phantom{-}1\\
		1 & \phantom{-}1 & -1 & \phantom{-}1 & \phantom{-}1 & \phantom{-}1 & -1 & \phantom{-}1 & \phantom{-}1 & \phantom{-}1 & -1 & -1 & -1 & \phantom{-}1 & \phantom{-}1\\
		1 & \phantom{-}1 & \phantom{-}1 & -1 & \phantom{-}1 & \phantom{-}1 & \phantom{-}1 & -1 & \phantom{-}1 & \phantom{-}1 & -1 & \phantom{-}1 & \phantom{-}1 & -1 & -1\\
		1 & \phantom{-}1 & \phantom{-}1 & \phantom{-}1 & -1 & \phantom{-}1 & \phantom{-}1 & \phantom{-}1 & -1 & \phantom{-}1 & \phantom{-}1 & -1 & \phantom{-}1 & -1 & \phantom{-}1\\
		1 & \phantom{-}1 & \phantom{-}1 & \phantom{-}1 & \phantom{-}1 & -1 & \phantom{-}1 & \phantom{-}1 & \phantom{-}1 & -1 & \phantom{-}1 & \phantom{-}1 & -1 & \phantom{-}1 & -1
	\end{bmatrix}.
\end{align*}

We have taken this example from \cite{bowlin-max-frus}. The major objective of \cite{bowlin-max-frus} was to find the maximum frustration index of a graph over all possible signings of the edges, but we use this example to show the effectiveness of the gain dependent bound obtained in \eqref{eq:bipartite-opt} for estimating $\lambda_1(\Phi)$, especially when the value of $a(\Phi)$ is large. In Table \ref{tab:bipartite-lambda-bounds}, we summarize the performance of all the bipartite bounds obtained in this paper, for both small and large values of $a(\Phi)$, in the upper and lower half of the table respectively. Note that for the Erd\H{o}s-Re\'yni graphs used for experiments, the $a(\Phi)$ value is intrinsically low. To get results for higher values of $a(\Phi)$, we use the $b(\Phi)$ values instead, which are typically high for the chosen graphs.

\begin{figure}[H]
	\centering
	\subfloat[$H_1 \:( n=10, p=0.70, m=17)$]{\includegraphics[width=5cm,scale=0.8]{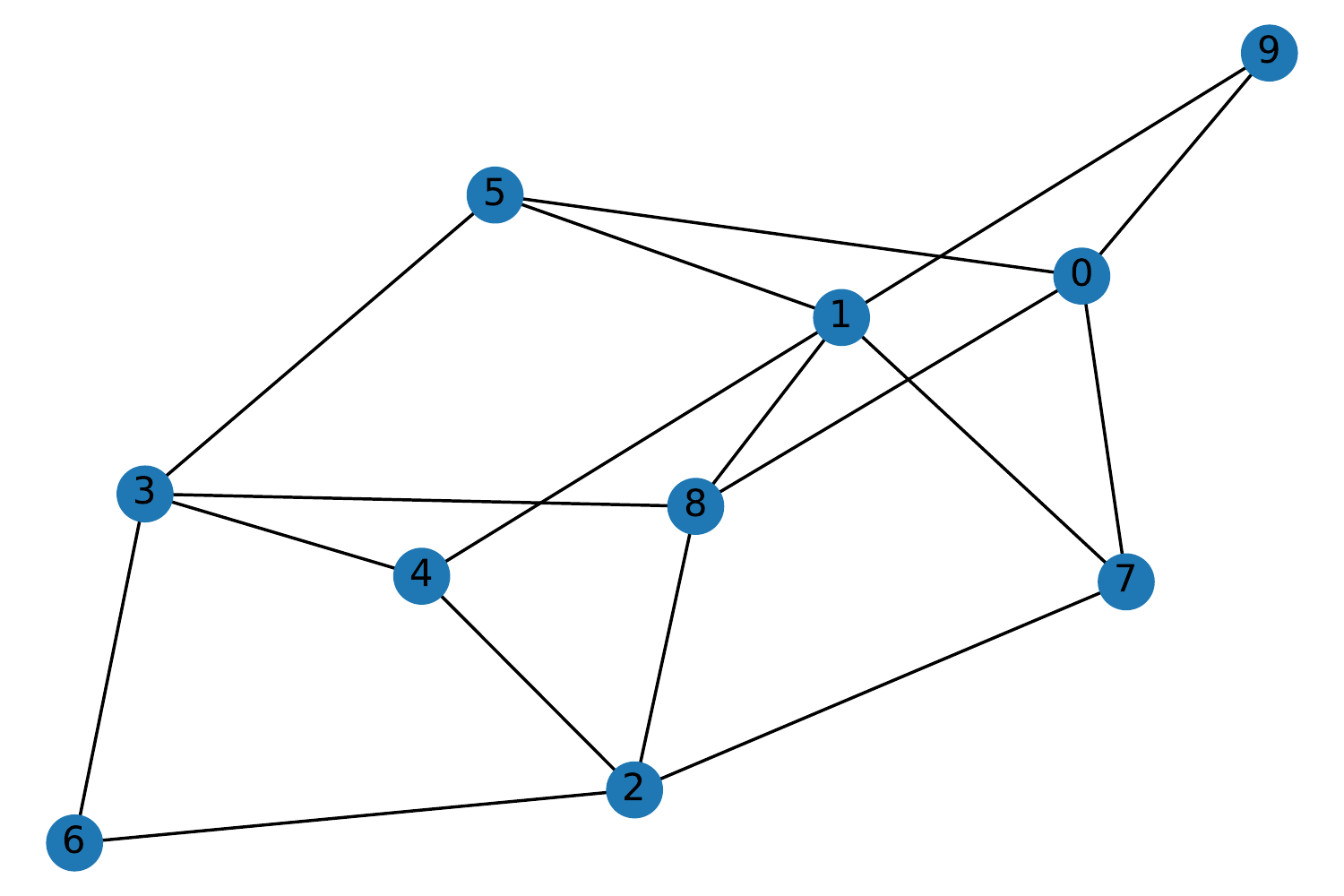}}
	\quad
	\subfloat[$H_2 \:( n=12, p=0.70, m=22)$]{\includegraphics[width=5cm,scale=0.8]{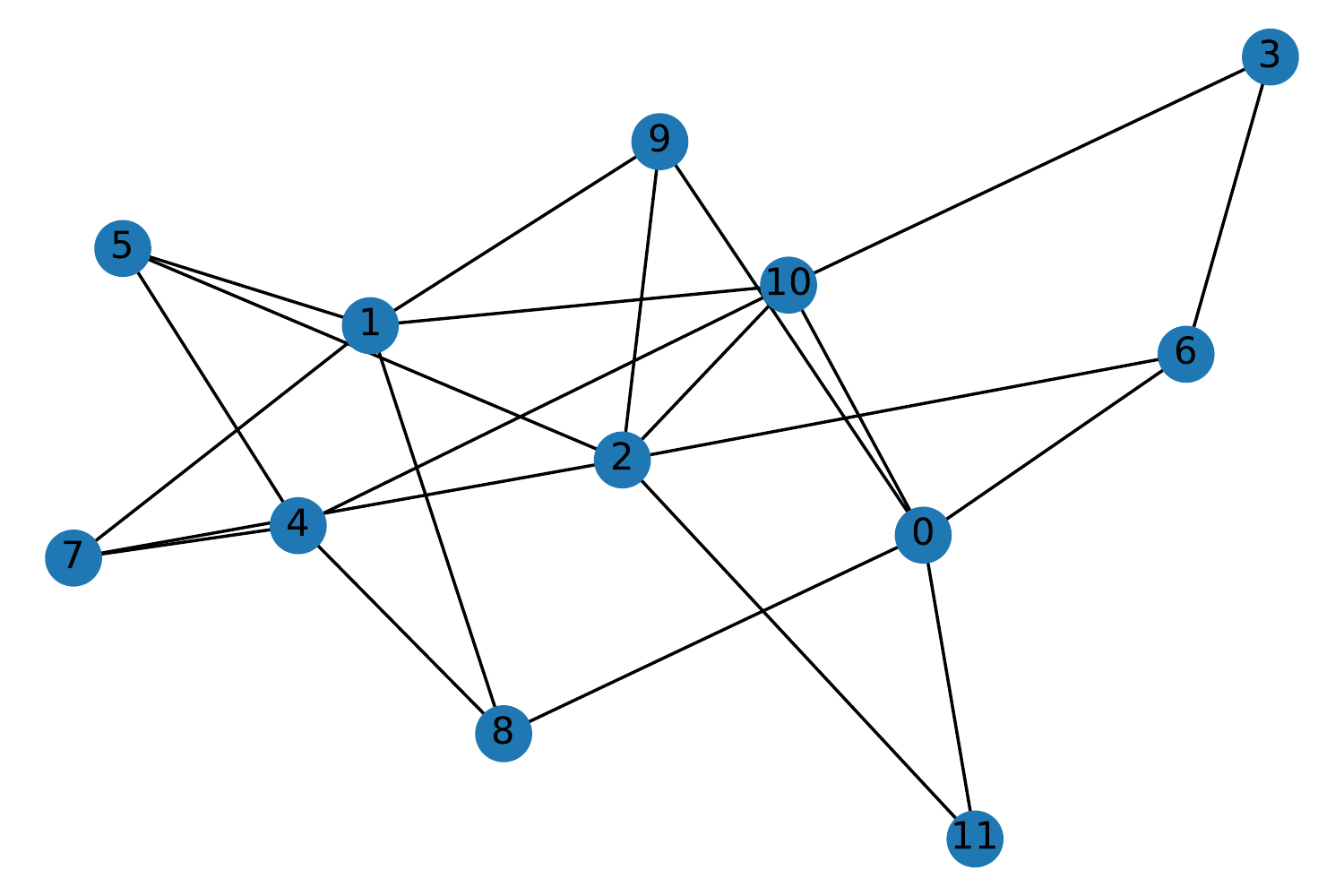}}
	\quad
	\subfloat[$H_3 \:( n=14, p=0.70, m=31)$]{\includegraphics[width=5cm,scale=0.8]{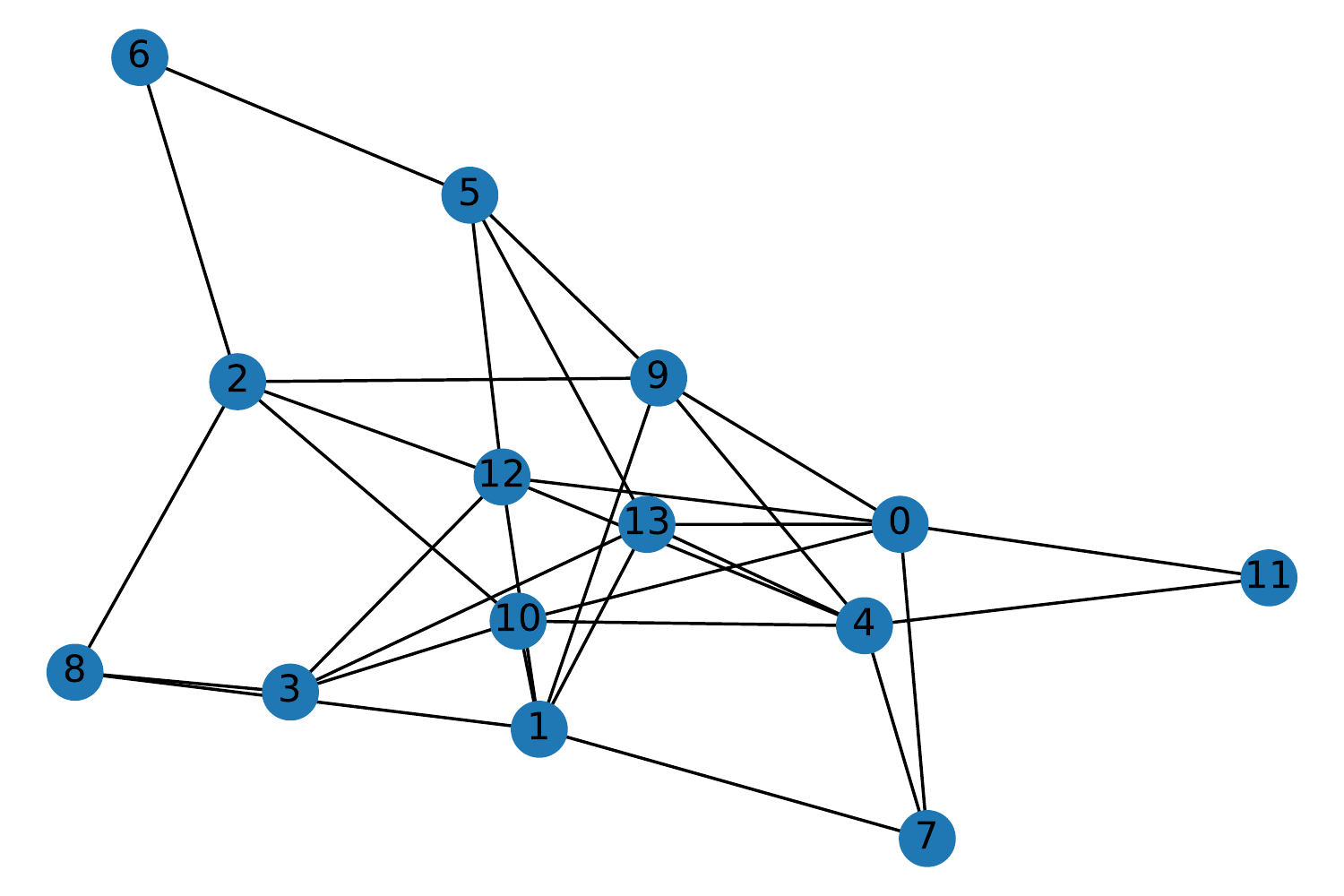}}
	\caption{\small Erd\H{o}s-Re\'yni Bipartite graphs. }
	\label{fig:bipartite-graphs}
\end{figure}

In Figure \ref{fig:bipartite-graphs}, we draw the Erd\H{o}s-Re\'yni bipartite graphs used in order to analyze the bound given in \eqref{eq:bipartite-opt}, and set the seed parameter to $0$ in \texttt{Networkx}. The bounds  are provided in Table \ref{tab:bipartite-lambda-bounds}. By looking at the highlighted entries in the following tables, we could infer that the bounds established in this work are better than some of the other bounds.

\newcolumntype{R}{>{$}r<{$}}  %
\newcolumntype{V}[1]{>{[\;}*{#1}{R@{\;\;}}R<{\;]}} %
\renewcommand{\arraystretch}{2.4}
\begin{table}[H]
	{\footnotesize
		\captionsetup{position=top} 
		\caption{\small Comparison of upper bound for $\lambda_1(\Phi)$ for bipartite graphs. }\vspace*{-1mm}\label{tab:bipartite-lambda-bounds}
		\begin{center}
			\subfloat{
				\begin{tabular}{|l|R|*4{p{16mm}}|} \Xhline{4\arrayrulewidth}
					\multicolumn{6}{|c|}{\textbf{\textsf{Upper Bounds for $\lambda_1(\Phi)$}}} \\ \hline
					
					\multicolumn{2}{|c|}{\textbf{\textsf{Bounds}}} &
					$H_1$ & $H_2$ & $H_3$ & $\substack{K_{5,15}\\ \eqref{eq: bipartite-adj}}$\\ \hline \hline
					
					\multicolumn{2}{|c|}{$\lambda_1(\Phi)$} & $0.018$ & $0.042$ & $0.050$ & $3.597$ \\ \Xhline{2\arrayrulewidth}
					
					\multicolumn{2}{|c|}{$a(\Phi)$} & $0.017$ & $0.025$ & $0.018$ & $0.613$  \\ \hline
					\multicolumn{2}{|c|}{$\frac{2m}{n}\cdot a(\Phi)$} & \hlgreen{$0.059$} & \hlgreen{$0.093$} & \hlgreen{$0.078$} & $4.600$ \\ \hline
					
					\multicolumn{2}{|c|}{\eqref{eq:bipartite-opt}} & \hlgreen{$0.059$} & \hlgreen{$0.093$} & \hlgreen{$0.078$} & \hlgreen{$3.982$}\\ \Xhline{2\arrayrulewidth}
					
					\multicolumn{2}{|c|}{$b(\Phi)$} & $0.958$ & $0.977$ & $0.991$ & $1.000$  \\ \hline
					
					\multicolumn{2}{|c|}{$\frac{2m}{n}\cdot b(\Phi)$} & $3.257$ & $3.581$ & $4.388$ & $7.500$\\ \hline
					
					\multicolumn{2}{|c|}{\eqref{eq:bipartite-opt}} & $2.818$ & $3.137$ & $3.874$ & $5.000$ \\ \hline

					\Xhline{4\arrayrulewidth}
			\end{tabular}}
		\end{center}
	}
	\vspace*{-7mm}
\end{table}

In Figure \ref{fig:erdos-reyni-graphs}, we draw the Erd\H{o}s-Re\'yni graphs, and set the seed parameter to $n$ (number of vertices in the random graph) in \texttt{Networkx}. We made comparisons between different bounds using these graphs as basis.
\begin{figure}[H]
	\centering
	\subfloat[$G_1 \:( n=6, p=0.55, m=8)$]{\includegraphics[width=5.5cm]{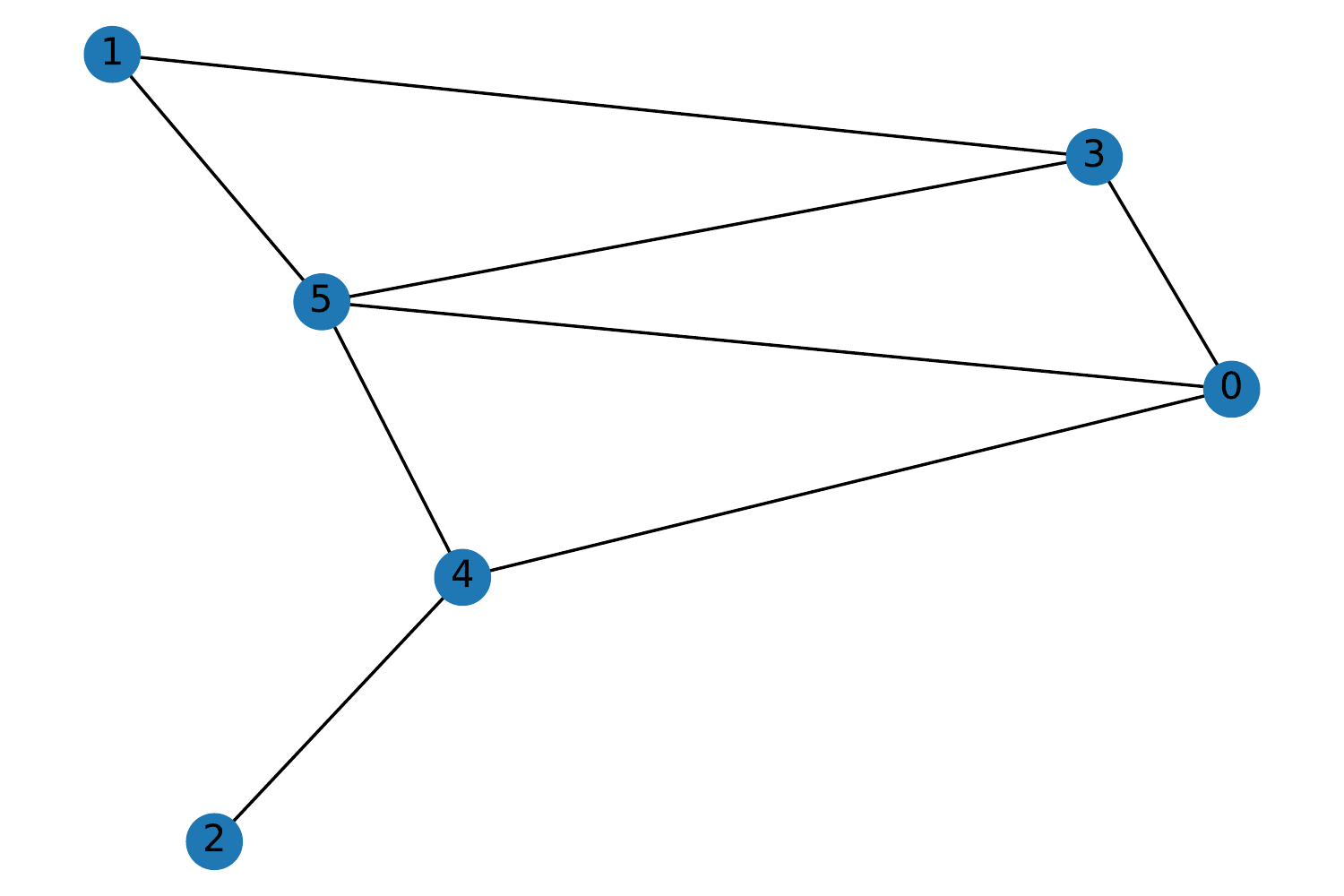}}
	\quad
	\subfloat[$G_2 \:( n=7, p=0.72, m=18)$]{\includegraphics[width=5.5cm]{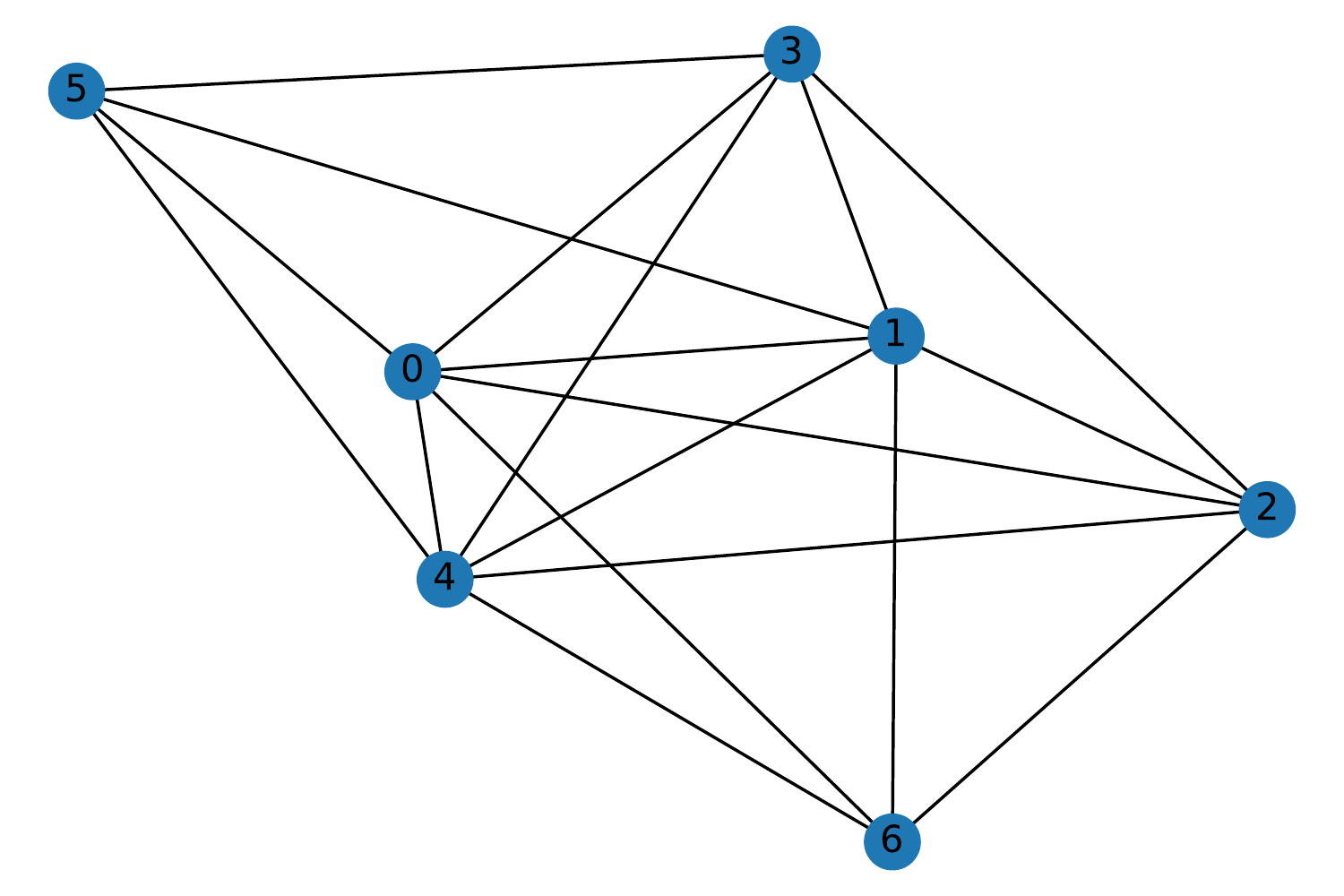}}
	\quad
	\subfloat[$G_3 \: (n=8, p=0.60, m=19)$]{\includegraphics[width=5.5cm]{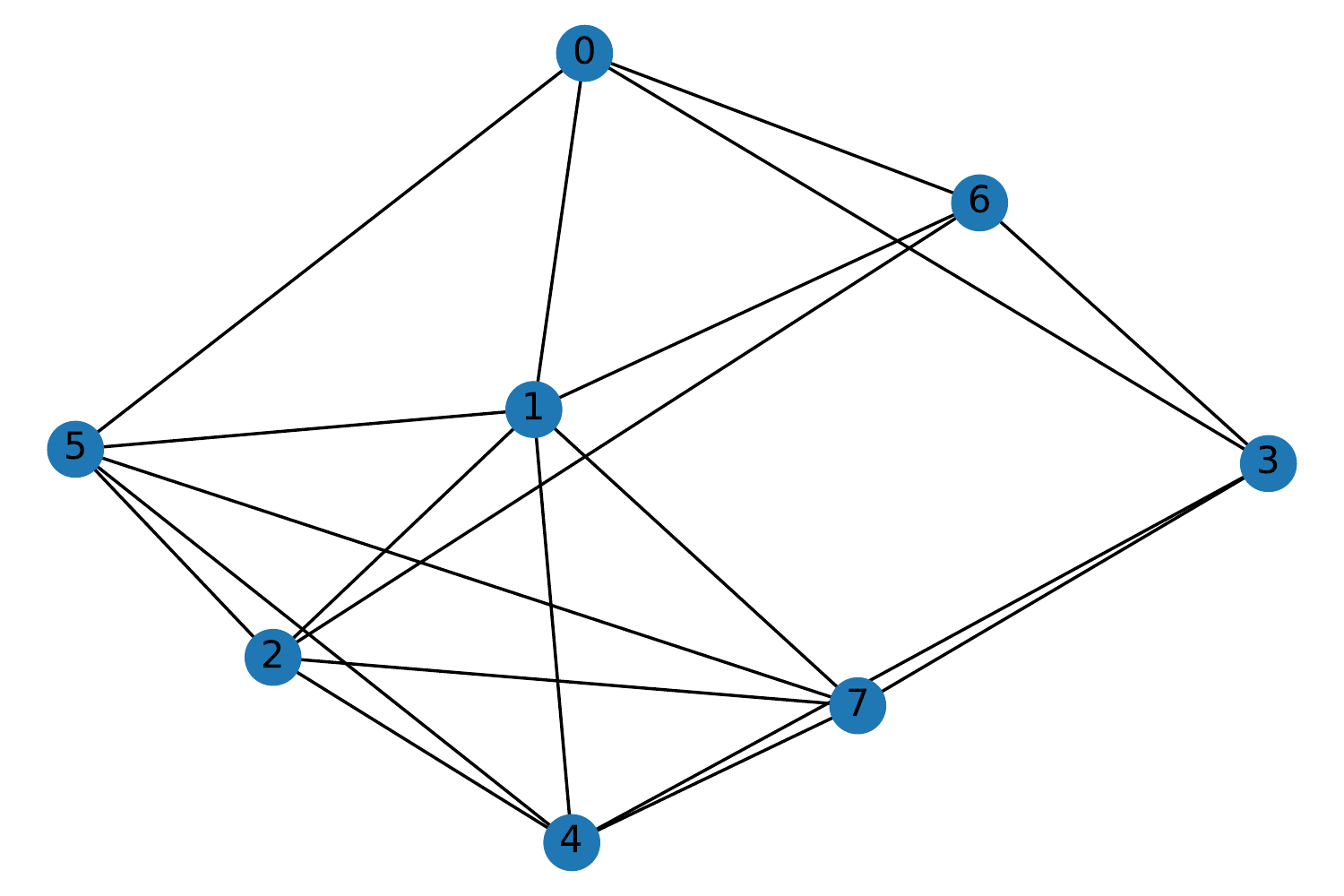}}
	\quad
	\subfloat[$G_4  \:(n=9, p=0.54, m=22)$]{\includegraphics[width=5.5cm]{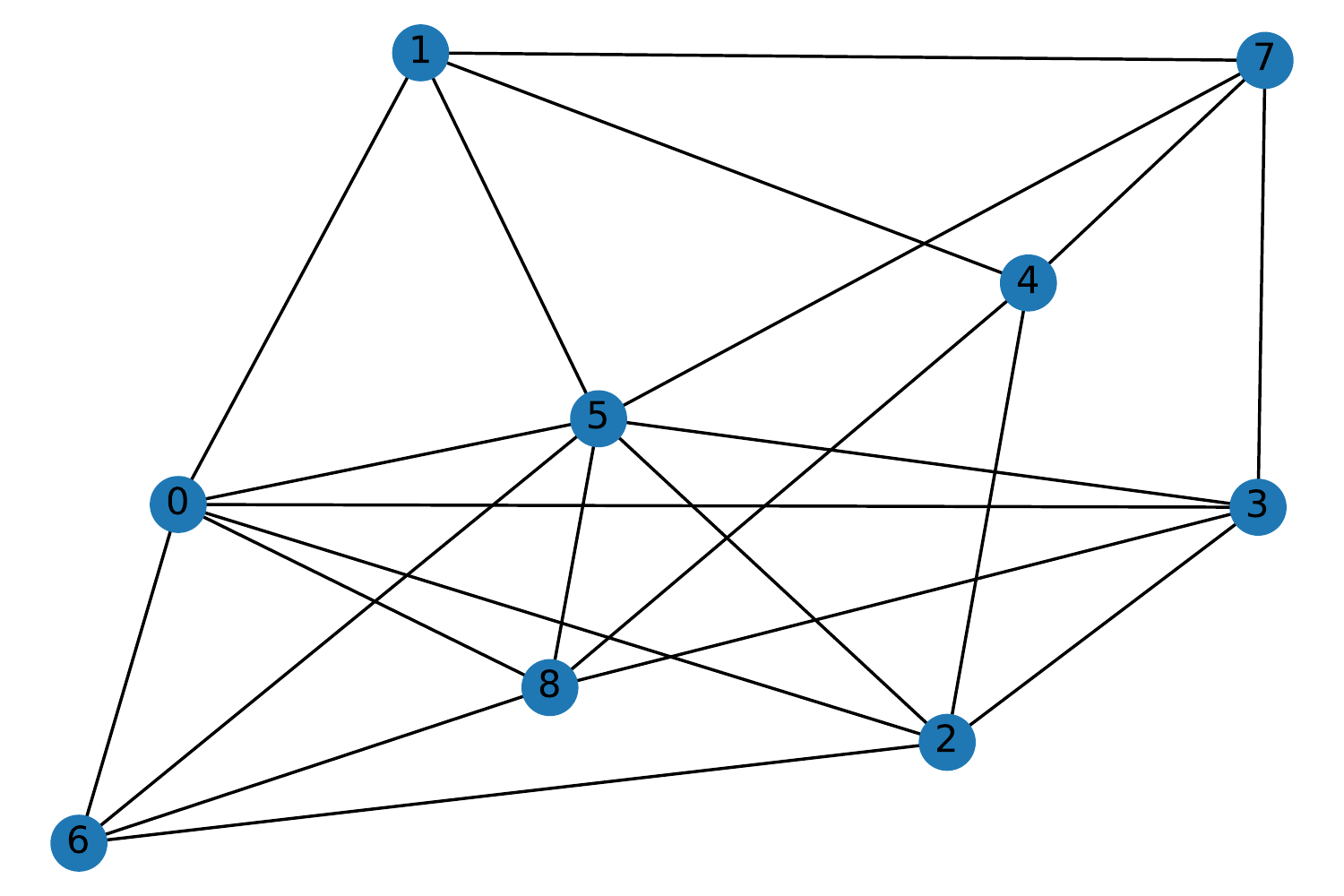}}
	\quad
	\subfloat[$G_5 \:( n=10, p=0.42, m=19)$]{\includegraphics[width=5.5cm]{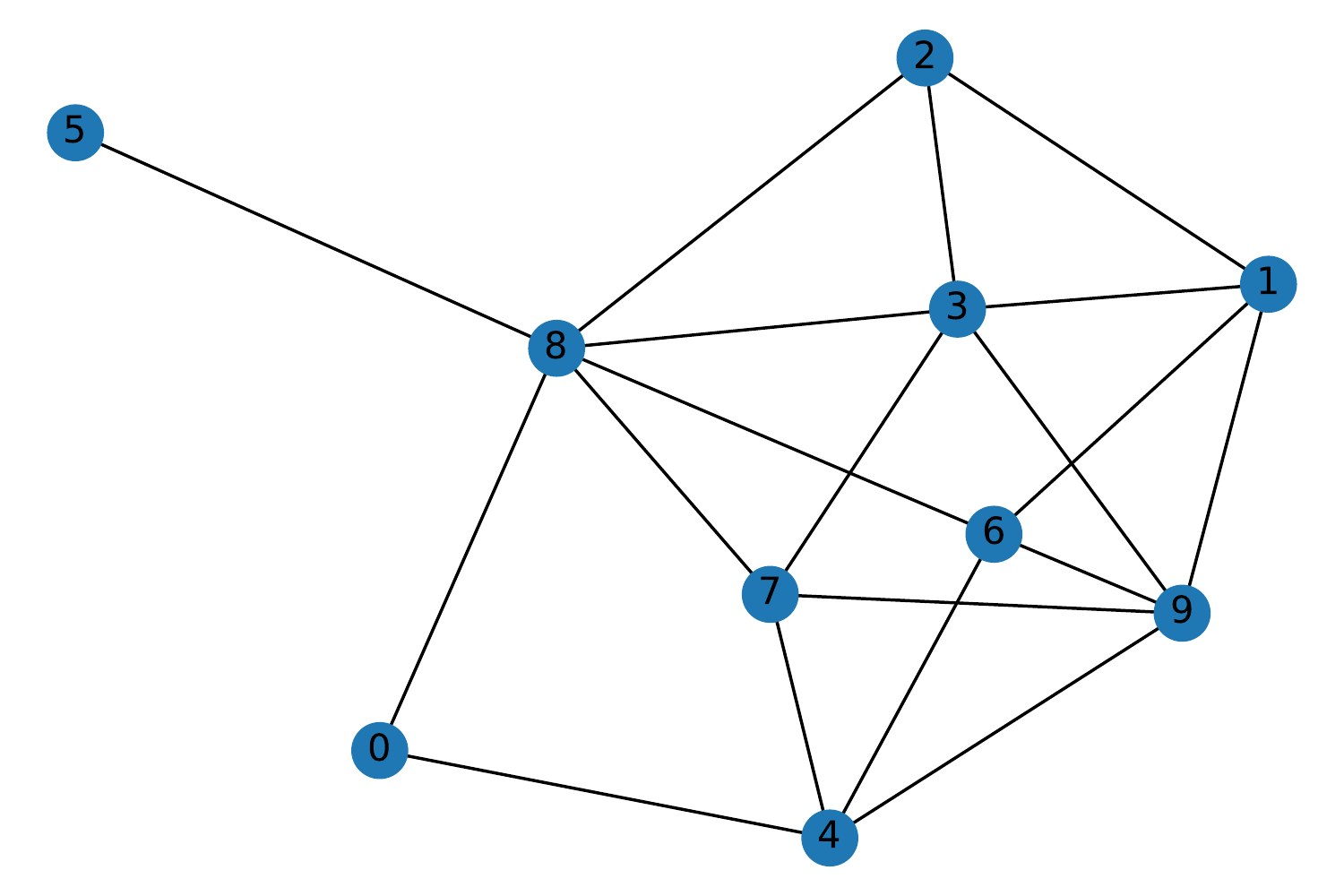}}
	\caption{\small Erd\H{o}s-Re\'yni graphs. }
	\label{fig:erdos-reyni-graphs}
\end{figure}Table \ref{tab:upper-bounds-for-least-eig} summarizes the upper bounds obtained in this paper for the least eigenvalue $\lambda_1(\Phi)$, for each of the Erd\H{o}s-Re\'yni graphs given in Figure \ref{fig:erdos-reyni-graphs}. The most noticeable bounds are the gain dependent bounds involving $a(\Phi)$ and $a_{\theta}(\Phi)$ quantities, which perform extremely well as compared to the existing gain independent bounds and the optimal degree bounds proven in this paper. It is also worth noting that the bound given in \eqref{eq:chromres-5-opt} which involves the quantity $a_{\theta}(\Phi)$ has superior performance across all instances of Erd\H{o}s-Re\'yni graphs.

\newcolumntype{R}{>{$}r<{$}}  %
\newcolumntype{V}[1]{>{[\;}*{#1}{R@{\;\;}}R<{\;]}} %
\renewcommand{\arraystretch}{2.8}
\begin{table}[H]
	{\footnotesize
		\captionsetup{position=top} 
		\caption{\small Comparison of upper bounds for $\lambda_1(\Phi)$ .}\vspace*{-1mm}\label{tab:upper-bounds-for-least-eig}
		\begin{center}
			\subfloat{
				\begin{tabular}{|c|R|*5{p{16mm}}|} \Xhline{4\arrayrulewidth}
					
					\multicolumn{7}{|c|}{\textbf{\textsf{Upper Bounds for $\lambda_1(\Phi)$}}} \\ \hline
					
					\multicolumn{2}{|c|}{\textbf{\textsf{Bounds}}} &
					$G_1$ & $G_2$ & $G_3$ & $G_4$ & $G_5$\\ \hline \hline
					
					\multicolumn{2}{|c|}{$\lambda_1(\Phi)$} & $0.013$ & $0.059$ & $0.05$ & $0.039$ & $0.032$ \\ \hline
					
					\multicolumn{2}{|c|}{$a(\Phi)$} & $0.03$ & $0.015$ & $0.017$ & $0.010$ & $0.014$ \\ \hline
					
					\multicolumn{2}{|c|}{$b(\Phi)$} & $1.1$ & $0.982$ & $1.041$ & $0.984$ & $0.977$ \\ \hline

					\multicolumn{2}{|c|}{$\frac{2m}{n}\cdot a(\Phi)$} & $0.075$ & \hlgreen{$0.079$} & $0.082$ & \hlgreen{$0.051$} & $0.054$ \\ \hline

					\multicolumn{2}{|c|}{\eqref{eq:chromres-5-opt}} & \hlgreen{$0.065$} & \hlgreen{$0.079$} & \hlgreen{$0.080$} & \hlgreen{$0.051$} & \hlgreen{$0.053$} \\
					\Xhline{2\arrayrulewidth}
					
					\multicolumn{2}{|c|}{$\min\limits_{v_s \sim v_t} \bigg(\frac{1}{2}\{d_s+d_t-2\}\bigg)$} & $1.000$ & $3.500$ & $3.000$ & $3.000$ & $2.000$ \\ \hline
					
					\multicolumn{2}{|c|}{$\min\limits_{v_s \sim v_t}\frac{1}{2}\Bigg(d_s+d_t- \sqrt{(d_s-d_t)^2+4}\Bigg)$} & $0.586$ & $3.382$ & $3.000$ & $3.000$ & $0.807$ \\ \hline

					\multicolumn{2}{|c|}{\eqref{eq:3d-degree-bound1-triangle}} & $1.013$ & $3.000$ & $2.005$ & $2.020$ & $2.012$ \\ \hline
					
					\multicolumn{2}{|c|}{\eqref{eq:3d-degree-bound2-triangle}} & $0.873$ & $2.814$ & $2.005$ & $2.020$ & $1.872$ \\ \hline
					
					\multicolumn{2}{|c|}{\eqref{eq:3d-degree-bound3-triangle}} & $1.013$ & $2.865$ & $2.005$ & $2.020$ & $1.827$ \\ \hline
					
					\multicolumn{2}{|c|}{\eqref{eq:3d-degree-bound4-triangle}} & $0.379$ & $3.004$ & $2.005$ & $2.020$ & $1.517$ \\ \hline
					
					\multicolumn{2}{|c|}{\eqref{eq:3d-degree-bound1-tree}} & $1.000$ & $3.333$ & $3.000$ & $3.000$ & $1.667$ \\ \hline
					
					\multicolumn{2}{|c|}{\eqref{eq:3d-degree-bound2-tree}} & $0.775$ & $3.293$ & $2.775$ & $2.775$ & $0.411$ \\ \hline
					
					\multicolumn{2}{|c|}{\eqref{eq:3d-degree-bound3-tree}} & $0.634$ & $3.219$ & $2.775$ & $2.775$ & $0.775$ \\ \hline
					
					\multicolumn{2}{|c|}{\eqref{eq:3d-degree-bound4-tree}} & $0.363$ & $3.268$ & $2.814$ & $2.814$ & $0.551$ \\
					
					\Xhline{4\arrayrulewidth}
			\end{tabular}}
		\end{center}
	}
	\vspace*{-7mm}
\end{table}

\newcolumntype{R}{>{$}r<{$}}  %
\newcolumntype{V}[1]{>{[\;}*{#1}{R@{\;\;}}R<{\;]}} %
\renewcommand{\arraystretch}{2}
\begin{table}[H]
	{\footnotesize
		\captionsetup{position=top} 
		\caption{\small Comparison of  different lower and upper bounds for $\lambda_n(\Phi)$. Here $r$ is defined in \eqref{eq: gershgorian-n}.}\vspace*{-1mm}\label{tab:bounds-for-max-eig}
		\begin{center}
			\subfloat{
				\begin{tabular}{|l|R|*5{p{16mm}}|} \Xhline{4\arrayrulewidth}
					\multicolumn{7}{|c|}{\textbf{\textsf{Lower Bounds for $\lambda_n(\Phi)$}}} \\ \hline
					
					\multicolumn{2}{|c|}{\textbf{\textsf{Bounds}}} &
					$\substack{G_1 \\
						(r=0.560)}$ & $\substack{G_2 \\
						(r=0.990)}$ & $\substack{G_3 \\
						(r=0.990)}$ & $\substack{G_4 \\
						(r=0.990)}$ & $\substack{G_5 \\
						(r=0.320)}$\\ \hline \hline
					
					\multicolumn{2}{|c|}{$\lambda_n(\Phi)$} & $5.19$ & $7.44$ & $7.59$ & $8.29$ & $7.70$ \\ \Xhline{2\arrayrulewidth}
					
					\multicolumn{2}{|c|}{$\Delta+1$} & $5$ & $7$ & $7$ & $8$ & $7$ \\ \hline
					
					\multicolumn{2}{|c|}{$\pmb{^\ddag}$ $\bigg(\max\limits_{i \in \{1,2,\dots,n\}} \big(\mathbf{L}^k(\Phi)_{ii}\big)\bigg)^{\frac{1}{k}}$} & $(11,5.002)$ & $(12,7.013)$ & $(7,7.014)$ & $(9,8.005)$ & $(6,7.081)$ \\ \hline
					
					\multicolumn{2}{|c|}{$\pmb{^\S}$\eqref{eq:trace-powers-lambda}} & $(29,5.002)$ & $(17,7.014)$ & $(17,7.021)$ & $(42,8.000)$ & $(17,7.018)$ \\ \hline
					

					\multicolumn{7}{|c|}{\textbf{\textsf{Upper Bounds for $\lambda_n(\Phi)$}}} \\ \hline \hline
					
					\multicolumn{2}{|c|}{$\lambda_n(-)$} & $6.242$ & $10.661$ & $9.892$ & $10.560$ & $8.868$ \\ \Xhline{2\arrayrulewidth}
					
					\multicolumn{2}{|c|}{$2\Delta$} & $8$ & $12$ & $12$ & $14$ & $12$ \\ \hline
					
					\multicolumn{2}{|c|}{$\max\limits_{v_i \sim v_j}\{d_i+d_j\}$} & $7$ & $12$ & $11$ & $13$ & $11$\\ \hline
					
					\multicolumn{2}{|c|}{$\max\limits_{v_i \in V(\Phi)}\{d_i+m_i\}$} & $6.750$ & $11.000$ & $10.667$ & $11.714$ & $9.400$ \\ \hline
					
					\multicolumn{2}{|c|}{$\max\limits_{v_i \in V(\Phi)}\{d_i+m_i^k\}$} & $6.750$ & $11.000$ & $10.667$ & $11.714$ & $9.400$ \\ \hline
					
					\multicolumn{2}{|c|}{$\max\limits_{v_i \in V(\Phi)}\{d_i+n_i^k\}$} & \hlgreen{$6.547$} & \hlgreen{$10.802$} & ${10.403}$ & \hlgreen{$11.338$} & $9.359$ \\ \hline
					
					\multicolumn{2}{|c|}{$\max\limits_{v_i \in V(\Phi)}\{d_i+l_i^k\}$} & $6.635$ & {$10.834$} & $10.470$ & $11.572$ & $11.362$ \\ \hline
					
					\multicolumn{2}{|c|}{$\max\limits_{v_i \sim v_j}\Bigg\{\frac{(d_i (d_i+m_i)+d_j(d_j+m_j))}{(d_i+d_j)}\Bigg\}$} & {$6.571$} & $11$ & \hlgreen{$10.364$} & {$11.385$} & \hlgreen{$9.300$} \\ \hline
					
					
					\Xhline{4\arrayrulewidth}
			\end{tabular}}
		\end{center}
	}
	\vspace*{-7mm}
\end{table}
\footnotesize{* In $\pmb{\ddag}$ and $\pmb{\S}$, we write the value of $k$ along with the bound, for which $\Delta+1$ bound is beaten.\\In rest of the $k$ involving bounds we write the minimum possible bound value for all values of $0 \leq k \leq 100$.}\\

\section{Conclusion}
{In this paper, we considered the notion of Frustration number and Frustration index for complex unit gain graphs, and established
	that they are lower bounded by the smallest eigenvalue, $\lambda_1(\Phi)$, of the gain Laplacian matrix. The fact that computing frustration index  is an NP-hard problem, makes this bound invaluable when it comes to the applicability of frustration index in real world applications \cite{aref2019signed}. Furthermore, we provided several optimal bounds for the smallest and largest eigenvalues of the gain Laplacian matrices, which are on their own crucial quantities for several problems arising in spectral graph theory and its applications, and more importantly, could be used to efficiently estimate the frustration index and the frustration number. It is well established that $\lambda_1(\Phi)$ is a measure of balance and a gain graph is balanced if and only if $\lambda_1(\Phi)=0$ \cite{reff2012spectral}, therefore optimal bounds for $\lambda_1(\Phi)$ are extremely useful in determining the balance of the underlying graph. As a major highlight of this paper, we provided several gain-dependent bounds for $\lambda_1(\Phi)$, which reduces to $0$ for the usual Laplacian and signless bipartite Laplacian. Not only these bounds are novel but also perform extremely well for the optimal choice of gains, and reduce to their existing counterparts present in the literature, for special choices of gains. We also showed limit convergence for two of our bounds to the largest eigenvalue, and obtained optimal extremal bounds by posing optimization problems to achieve the best possible bounds.}

\section*{Acknowledgments}

We are indebted to the referee for the comments and detailed suggestions, which helped us to improve the manuscript substantially. M. Rajesh Kannan would like to thank the Department of Science and Technology, India, for financial support through the projects MATRICS (MTR/2018/000986) and Early Career Research Award (ECR/2017/000643).
	
	\bibliographystyle{plain}
	\bibliography{kannan-kumar-pragada}

\end{document}